\documentclass[reqno, a4paper, 10pt]{amsart}
\usepackage{amssymb,url, color, mathrsfs}
\usepackage[colorlinks=true, bookmarks=true, pdfstartview=FitH, pagebackref=true, linktocpage=true]{hyperref}
\usepackage[short,nodayofweek]{datetime}
\usepackage[pagewise,running,mathlines,displaymath, switch]{lineno}
\usepackage{times}
\usepackage{indentfirst, tabularx}
\usepackage[table]{xcolor}
\usepackage{float}

\theoremstyle{plain}
\numberwithin{equation}{section}
\newtheorem{theorem}{Theorem}[section]
\newtheorem{lemma}{Lemma}[section]

\newtheorem{proposition}{Proposition}[section]
\theoremstyle{remark}

\newcommand{\ps}{p_{\mathsf S}}
\newcommand{\pc}{p_{\mathsf{JL}}}

\definecolor{brown}{rgb}{0.5,0,0}
\definecolor{backgroundcolor}{rgb}{0.98, 0.92, 0.73}

\def\R{\mathbf R}

\def\P{\mathcal P}

\newif\ifprint
\ifprint
\else
\fi

\author[N.T. Tai]{Nguyen Tien Tai}
\address[N.T. Tai]{Department of Mathematics\\
College of Science, Vi\^{e}t Nam National University\\
334 Nguy\^{e}n Trai Street, H\`{a} N\^{o}i, Vi\^{e}t Nam.}
\email{\href{mailto: N.T. Tai <nttai.hus@vnu.edu.vn>}{nttai.hus@vnu.edu.vn}}
\email{\href{mailto: N.T. Tai <nguyentientai.vnu@gmail.com>}{nguyentientai.vnu@gmail.com}}

\begin{document}
\allowdisplaybreaks

\setpagewiselinenumbers
\setlength\linenumbersep{100pt}

\title[On the radial entire solutions of  $(-\Delta)^3 u=u^p$ in $\R^n$ ]
{On the asymptotic behavior  of radial entire solutions for the equation  $(-\Delta)^3 u=u^p$ in $\R^n$ }

\begin{abstract}
Our main task in this note is to prove the existence and to classify the exact growth at infinity of radial positive $C^6$-solutions of $(-\Delta )^3 u = u^p$ in $\R^n$, where $n\geqslant 15$ and $p$ is bounded from below by the sixth-order Joseph-Lundgren exponent.  Following the main work of Winkler, we introduce the sub- and super-solution method and comparision principle to conclude the asymptotic behavior of solutions.

\end{abstract}

\date{\bf \today \; at \, \currenttime}



\maketitle


\section{Introduction}
In this article, we are  interest in studying  the exact growth at infinity of  positive solutions  of the following sixth-order Lane--Emden equation
\begin{equation}\label{triharmonic}
(-\Delta)^3 u = u^p 
\end{equation}
in $\R^n$, where $n\geqslant 7$ and $p>1$.  While equations  of form \eqref{triharmonic} involving Laplacian and bi-Laplacian have  been attracted by many mathematicians over the last few decades since its root in conformal geometry and blow-up theory, equations involving tri-Laplacian of the form \eqref{triharmonic} have just started capturing attention recently.  The motivation of working on equation \eqref{triharmonic} goes back to a recent paper by Lou, Wei and Zou \cite{LWZ16} in which a complete classification of finite Morse index and stable solutions to the equation 
\[
(-\Delta)^3 u =|u|^{p-1}u
\]
in $\R^n$. Furthermore, counterparts of \eqref{triharmonic} involving negative exponents such as the equation
\begin{equation}\label{R3triharmonic}
\Delta^3 u =- u^{-3} 
\end{equation}
in $\R^3$ and the equation
\begin{equation}\label{R5triharmonic}
\Delta^3 u = u^{-11}  
\end{equation}
in $\R^5$ were also studied in \cite{DN17} for \eqref{R3triharmonic} and in \cite{FX13} for \eqref{R5triharmonic}. In order to the necessarily of carrying research on the sixth-order Lane--Emden equation \eqref{triharmonic}, let us briefly mention some known results for similar equations with lower order of Laplacian. 

First, for the well--known the second-order Lane--Emden equation  
\begin{equation}\label{harmonic}
-\Delta u = u^{p} 
\end{equation}
in $\R^n$ with $n\geqslant 3$ and $p>1$, the following two exponents known as the Sobolev exponent given by $\ps(2,n):=(n+2)/(n-2)$ and the Joseph--Lundgren exponent given by 
\begin{equation}\label{Jlexponent2}
\pc(2,n)= 
\begin{cases}
\frac{(n-2)^2-4n+8\sqrt{n-1}}{(n-2)(n-10)} \quad& \text{if } n\geqslant 11,\\
\infty \quad& \text{if } n\leqslant 10
\end{cases}
\end{equation}
play an important role, see \cite{GS81, GNW92, Wan93}. To be more precise,
if $\ps(2,n)$, any positive radial $\textit{C}^2$-solution $u$ enjoys the following asymptotic behavior
\[ 
\lim_{|x| \to \infty} \frac{u(x)}{L_2|x|^{-m_2}} = 1
\]
at infinity where 
\[ 
\begin{cases}
m_2=\frac{2}{p-1},\\
 L_2=(m_2(n-2-m_2))^{\frac{1}{p-1}}.
\end{cases}
\] 
Furthermore,  if $p\geqslant \pc(2,n)$, given any positive number $b$, there exists a unique regular radial solution $u$ of \eqref{harmonic} such that either 
\begin{equation}\label{CriticalLowerTerm1}
\lim_{|x| \to \infty} \frac{|x|^{m_2+\lambda_1}}{\log |x|}  \left(u(x) -L_2|x|^{-m_2}\right) =b, 
\end{equation}
or
\begin{equation}\label{SuperCriticalLowerTerm1}
\lim_{|x| \to \infty} |x|^{m_2+\lambda_1} \left(u(x) -L_2|x|^{-m_2}\right) =b,
\end{equation}
depending on either $p=\pc(2,n)$ or not. Here, $\lambda_1$ is the smallest root of the quadratic equation
\[
(m_2+\lambda)(n-2-m_2-\lambda)+p L_2^{p-1}=0,
\]
which has two real distinct roots if and only if $p >\pc(2,n)$ while it has double root when $p=\pc(2,n)$.

For the corresponding fourth-order equation
\begin{equation}\label{biharmonic}
\Delta^2 u =u^p 
\end{equation}
in $\R^n$  with $n\geqslant 5$ and $p>1$, an  analouge result  was also obtained  in \cite{GG06,Win10}. In this context, the Sobolev exponent denoted by $\ps(4,n)$ is given by   $\ps(4,n) :=(n+4)/(n-4)$ and the Joseph-Lundgren exponent denoted by $\pc(4,n)$ 
  is  the unique root greater than $\ps(4,n)$ of   the algebraic equation
\begin{align*}
&-(n-4)(n^3-4n^2-128n+256)(p-1)^4+128(3n-8)(n-6)(p-1)^3\\
&\qquad \qquad+256(n^2-18n+52)(p-1)^2-2048(n-6)(p-1)+4096=0.
\end{align*}
A formula for $\pc(4,n)$ can be precisely write down as follows
\begin{equation}\label{JLexponent4}
\pc(4,n) =
\begin{cases}
\frac{n+2-\sqrt{n^2+4-n \sqrt{n^2-8n+32}}}{n-6-\sqrt{n^2+4-n \sqrt{n^2-8n+32}}}   \quad& \text{if } n \geqslant 13,\\
\infty \quad& \text{if } n \leqslant 11. 
\end{cases}
\end{equation}
Having $\ps(4,n)$ and $\pc(4,n)$ at hand, it was proved that if $p>\ps(4,n)$, then the following asymptotic behavior for any radial entire solution $u$ is well known
\[ 
\lim_{|x| \to \infty} \frac{u(x)}{L_4|x|^{-m_4}} = 1,
\]
where
\[
\begin{cases}
m_4=\frac{4}{p-1},\\
 L_4=(m_4(m_4+2)(n-2-m_4)(n-4-m_4))^{\frac{1}{p-1}}.
\end{cases}
\]
Moreover, when $p\geqslant \pc(4,n)$,  for any $b>0$, Eq. \eqref{biharmonic} admits a unique  radially symmetric solution $u$ satisfying either
\begin{equation}\label{CriticalLowerTerm2}
\lim_{|x| \to \infty} \frac{|x|^{m_4+\lambda_2}}{\log |x|}  \left(u(x) -L_4|x|^{-m_4}\right) =b, 
\end{equation}
or
\begin{equation}\label{SuperCriticalLowerTerm2}
\lim_{|x| \to \infty} |x|^{m_4+\lambda_2} \left(u(x) -L_4|x|^{-m_4}\right) =b,
\end{equation}
depending on either $p=\pc(4,n)$ or not. Here, $\lambda_2$ is the smallest positive root of the quartic equation
\[
(m_4+\lambda)(m_4+\lambda+2)(n-2-m_4-\lambda)(n-4-m_4-\lambda)+p L_4^{p-1}=0,
\]
which has four real  roots ordered as 
\[ \lambda_1 < 0 <\lambda_2 \leqslant \lambda_3 <\lambda_4.\]
In addition, $\lambda_2=\lambda_3$  if and only if $p=\pc(4,n)$.

Following the well known results presented above, we are interested in asymptotic behavior for solutions of \eqref{triharmonic} in the sense of \eqref{CriticalLowerTerm1}-\eqref{SuperCriticalLowerTerm1} for \eqref{harmonic} and \eqref{CriticalLowerTerm2}-\eqref{SuperCriticalLowerTerm2} for \eqref{biharmonic}. In order to go further, for Eq. \eqref{triharmonic} we first denote by $\ps(6,n)$ the Sobolev exponent given by $\ps(6,n):=(n+6)/(n-6)$  and by $\pc(6,n)$ the Joseph--Lundgren exponent (see \cite{HR17,LWZ16}) given by 
\begin{equation} \label{JLexponent6}
\pc(6,n)=
\begin{cases}
\frac{(n+4)\sqrt{3} - \sqrt{\sqrt[3]{K_0+K_1}+ \sqrt[3]{K_0-K_1} + 3n^2+32 }}{(n-8)\sqrt{3} - \sqrt{\sqrt[3]{K_0+K_1} +  \sqrt[3]{K_0-K_1} + 3n^2+32 }} 
 \quad& \text{if } n \geqslant 15,\\
\infty \quad& \text{if } n \leqslant 14,  
\end{cases}
\end{equation}
where 
\[
2K_0 =-27n^6+324 n^5-756n^4-2592 n^3 + 25776 n^2 +5184 n -23744,
\]
and 
\[
2K_1 = \sqrt{(2K_0)^2 - 4(192n^2+256)^3}.
\]

To seek for  the exact growth at infinity of radial solutions of \eqref{triharmonic},  
let us go back to the 
the work of Winkler \cite{Win10}.  To study  the existence  of solutions,  the author  transformed \eqref{biharmonic} into  the Lane Emden type system including two second-order elliptic equations  as follows
\begin{equation}\label{LESystem1}
\left\{
\begin{split}
-\Delta u &= v,\\
-\Delta v &=u^p
\end{split}
\right.
\end{equation}
in $\R^n$ and considers the associated parabolic problem 
\begin{equation}\label{LESystem2}
\left\{
\begin{split}
u_t &= \Delta u + v, \\
v_t &=\Delta v +|u|^{p-1}u 
\end{split}
\right.
\end{equation}
in  $\R^n \times (0,\infty)$ with the initial datas are sub-solutions of \eqref{LESystem1}.  Here, the existence of radial entire solutions of \eqref{LESystem2} comes from  the fact that it is of cooperative type. Furthermore, thanks to the  the comparision principle stated  in \cite[Prop. 52.21]{QS07}, it has been shown that solutions of \eqref{LESystem2} are bounded by sub- and super-solutions of \eqref{LESystem1}. Then, with the help from the monotonicity property and parabolic regularity theory, the conclusion follows that there exists the steady state positive solutions of \eqref{LESystem2}, i.e  the stationary problem \eqref{LESystem1}.

 Inspired by the main result of \cite{Win10}, in the present paper, we study the exact growth at infinity of radially symmetric solutions to \eqref{triharmonic} under the restriction $p\geqslant \pc(6,n)$. First, for the exponent $p$ higher than $\pc(6,n)$, we obtain the following result.
\begin{theorem}\label{MainThmSupercriticalCase}
Let $n\geqslant 15$ and  $p >\pc(6,n)$. Then, for given $b>0$, there exists a classical radially symmetric solution $u$ of \eqref{triharmonic} such that
\[
u(x)=L|x|^{-m}-b|x|^{-m-\lambda_3}+o(|x|^{-m-\lambda_2}) \quad \text{as} \quad |x|\to +\infty,
\]
where $m=6/(p-1)$, 
\[ 
L=(m(m+2)(m+4)(n-2-m)(n-4-m)(n-6-m))^{\frac{1}{p-1}}
\]
 and $\lambda_3$ is the smallest positive root among six real roots of the following polynomial
\[
P(\lambda) = (m+\lambda)(m+\lambda+2)(m+\lambda+4)(n-2-m-\lambda) (n-4-m-\lambda)(n-6-m-\lambda)  - pL^{p-1}.
\]
\end{theorem}

Second, for the exponent $p$ equal to $\pc(6,n)$, we also obtain an asymptotic behavior of solutions of \eqref{triharmonic} which involves a logarithmic correction. The precise statement of this case is as follows. 
\begin{theorem}\label{MainThmCriticalCase}
Let $n\geqslant 15$ and  $p =\pc(6,n)$. Then, for given $b>0$, there exists a classical radially symmetric solution $u$ of \eqref{triharmonic} such that
\[
u(x)=L|x|^{-m}-b|x|^{-m-\lambda_3} \log|x|+o(|x|^{-m-\lambda_2}) \quad \text{as} \quad |x|\to +\infty,
\]
where the notations still remain as Theorem \ref{MainThmSupercriticalCase}.
\end{theorem}

To conclude this section, we briefly show the organization of this note as follows. In the second section, several  notations and conventions shall be given. Furthermore, basic properties of the exponent $\pc(6,n)$ as well as the roots of polynomial $P$ appearing in the statement of Theorem \ref{MainThmSupercriticalCase} shall also be discussed.  In the third and fourth section, as routine, we  transform the sixth-order \eqref{triharmonic} into the system of three second-order equations
 \begin{equation}\label{LaneEmdenSystem}
\left\{
\begin{split}
-\Delta u &= v  &\quad\text{in } \R^n,\\
-\Delta v &= w  &\quad\text{in } \R^n,\\
-\Delta w &= u^p &\quad\text{in } \R^n
\end{split}
\right.
\end{equation}
  and contruct sub- and super-solutions  in the supercritical and critical cases, respectively. Finally, we prove Theorem  \ref{MainThmSupercriticalCase}  and  \ref{MainThmCriticalCase} in the last section.

\section{Preliminaries}

From now on, for  simplicity, we use $\ps$ and $\pc$ instead of $\ps(6,n)$ and $\pc(6,n)$. Besides, we also use the two notations $L$ and $m$ appearing in the statement of  Theorem \ref{MainThmSupercriticalCase}. By a direct computation, it is not hard to see that the following function
\[
u_{\infty}(x) =L|x|^{-m}, \quad x \in \R^n\setminus\{0\},
\]
 is the singular solution of \eqref{triharmonic} when $n\geqslant 7$ and $p>n/(n-6)$. Furthermore, linearizing \eqref{triharmonic} around  its singular solution $u_{\infty}$ leads us to the fact that
\begin{equation}\label{CharacteristicsEq}
(m+\lambda)(m+\lambda+2)(m+\lambda+4)(n-2-m-\lambda)(n-4-m-\lambda)(n-6-m-\lambda) = pL^{p-1}
\end{equation}
is the characteristic equation of \eqref{triharmonic}. Consequently, investigating all   zero point of this equation \eqref{CharacteristicsEq} is crucial. We have the following lemma. 
\begin{proposition}
Under the assumption $p \geqslant \pc$, we have the following  claims
\begin{enumerate}
\item[(a).] 
For $n\geqslant 15$, the polynomial
\begin{equation}\label{DegreeQ}
\begin{split}
Q(m)&=(m+6)(m+2)(m+4)(n-2-m)(n-4-m)(n-6-m)\\
&\qquad \qquad \qquad \qquad-\frac{(n-6)^2(n-2)^2(n+2)^2}{64}
\end{split}
\end{equation}
only has a unique zero point $\pc$, which is larger  than $\ps$ and $Q(m)$ is strictly negative as  $p>\pc$.
\item[(b).] Eq. \eqref{CharacteristicsEq} admits 6 real roots, denoted by $\lambda_i$ with $1\leqslant  i \leqslant 6$, ordered according to 
\[
\lambda_1<\lambda_2<0< \lambda_3 \leqslant  \lambda_4< \lambda_5<\lambda_6.
\]
The equality $\lambda_3=\lambda_4$ holds only in the case $p=\pc(6,n)$.
\end{enumerate}
\end{proposition}
\begin{proof} 
\noindent\textbf{Part (a)}. See \cite[Prop. 1.1]{HR17}.

\noindent\textbf{Part (b)}. With the change of variable $l=m+\lambda$, we rewrite  \eqref{CharacteristicsEq} as 
\begin{equation}\label{Degree6Poly2}
P(l) = l(l+2)(l+4)(n-2-l)(n-4-l)(n-6-l) - pL^{p-1}.
\end{equation}
Set $x=l(n-6-l)$, then \eqref{Degree6Poly2} transforms into 
\begin{equation}\label{Degree3Poly}
P_1(x)= x^3+bx^2+cx-d,
\end{equation}
where 
\begin{equation*}
\left\{
\begin{split}
b&=6n-16,\\
c&=8(n-2)(n-4),\\
d&=pL^{p-1}=(m+6)(m+2)(m+4)(n-2-m)(n-4-m)(n-6-m).
\end{split}
\right.
\end{equation*}
It is not hard to see that $b,c$ and $d$ are positive.  Note that 
\[
P_1'(x) = 3x^2+2bx+c
\]
has two distinct roots. Hence, $P_1$ admits three distinct real roots $x_1,x_2$ and $x_3$. Due to Vieta's formula, it must be hold one of them is positive and the others are negative. Moreover, we have 
\[\begin{split}
&P_1(m(n-6-m))\\
&= (n-6-m)\left[ \begin{split}
&(-6m^2-36m-48)n^2+(12m^3+108m^2+312m+288)n\\
&-6m^4-72m^3-312m^2-576m-384
\end{split} \right]\\
&=-6(m+2)(m+4)(n-2-m)(n-4-m)(n-6-m)\\
&<0,
\end{split}\]
and
\[
P_1\left( \frac{(n-6)^2}{4} \right) = \frac{(n-6)^2(n-2)^2(n+2)^2}{64} - d \geqslant 0.\]
from \textbf{Part(a)}. Hence, without generality, we can assume that 
\begin{equation}\label{OrderXi} x_1<x_2<0<m(n-6-m)<x_3 \leqslant \frac{(n-6)^2}{4} .\end{equation}
For each $x_i$, we have to solve 
\[
l(n-6-l)=x_i. 
\]
Thanks to \eqref{OrderXi}, three above quadratic  equations admits six real roots 
 \[ l_i =\frac{n-6-\sqrt{(n-6)^2-4x_i}}{2}, \]
and 
\[ l_{7-i} = \frac{n-6+\sqrt{(n-6)^2-4x_i}}{2},\]
for $ i=1,2,3$, which are ordered as below 
\begin{equation}\label{OrderLi}
  l_1<l_2<0< l_3 \leqslant  l_4< l_5<l_6.
\end{equation}
The equality holds when the equation 
\begin{equation}\label{EquationX3}
 l(n-6-l) =x_3
\end{equation}
admits repeated roots $l_3=l_4=(n-6)/2$, i.e $x_3=(n-6)^2/4$, or $p=\pc$. In order to complete \textbf{Part(b)}, we have to verify  that two roots of  \eqref{EquationX3} are larger than $m$. That is equivalent to saying that the polynomial 
\[ \lambda^2-(n-6-2m)\lambda + x_3 - m(n-6-m) \]
has two positive roots. However, this is  apparent  from the definition of $m$ and \eqref{OrderXi}. Therefore, the ordering \eqref{OrderLi} can be improved to 
\[
 l_1<l_2<m< l_3 \leqslant  l_4< l_5<l_6.
\] 
That is the conclusion we are looking for.
\end{proof}

From now on, we always denote 
\[ l=m+\lambda_3, \]
where $\lambda_3$ is the smallest root of \eqref{CharacteristicsEq}.  Before ending this section, let us recall the following  useful lemma whose proof can be found in \cite[Lemmas 3.2, 3.4]{Win10}. 
\begin{lemma}\label{LemmaBinomialIne}
Let $p>0$ be given, there exists $C_p >0$ such that 
\begin{equation}\label{BinomialIne}
1-pz \leqslant (1-z)^p \leqslant 1-pz +C_p z^2
\end{equation}
for all $0\leqslant z\leqslant 1$.
\end{lemma}

\section{The supercritical case $p > \pc$ }

This section is devoted to construction of sub- and super-solutions to \eqref{LaneEmdenSystem} in the supercritical case $p>\pc$. For clarity, we divide our construction into two subsections. 
\subsection{Construction of a sub-solution \eqref{LaneEmdenSystem}}\label{subsec-supercritical-sub}

Our first aim is to construct an entire sub-solution $(\underline u, \underline v, \underline w)$ of \eqref{LaneEmdenSystem} when $p>\pc$, which is equal to zero near the origin. Then, we will discuss the explicit formula  of $\underline u, \underline v, \underline w$  in the outer consistent region.  Let   $b>0$ be given, we define the following functions,
\begin{equation}\label{OuterSubSolution}
\left\{
\begin{split}
\underline u_{\rm out}(r)&:= Lr^{-m}-br^{-l}, \\
\underline v_{\rm out}(r)&:= m(n-2-m)Lr^{-m-2}- l(n-2-l)br^{-l-2},\\
\underline w_{\rm out}(r)&:= m(m+2)(n-2-m)(n-4-m)Lr^{-m-4}\\
&\qquad \qquad- l(l+2)(n-2-l)(n-4-l)br^{-l-4}.
\end{split}
\right.
\end{equation}
From the definition, it is worth noticing that
\begin{equation}\label{remark}
\begin{cases}
\underline u_{\rm out}(r) >0 \quad& \text{if and only if} \quad r>\underline r_1, \\
\underline v_{\rm out}(r) >0 \quad& \text{if and only if} \quad r>\underline r_2,\\\underline w_{\rm out}(r) >0 \quad& \text{if and only if} \quad r>\underline r_3, 
\end{cases}
\end{equation}
where
\begin{equation}\label{UnderlineR0}
\underline{r}_1 := \left( \frac{b}{L}\right)^{\frac{1}{l-m}}, \quad \underline r_2:= \left(\frac{l(n-2-l)b}{m(n-2-m)L} \right)^{\frac{1}{l-m}},
\end{equation}
and
\begin{equation}\label{UnderlineR2}
\underline r_3:=\left(\frac{l(l+2)(n-2-l)(n-4-l)b}{m(m+2)(n-2-m)(n-4-m)L} \right)^{\frac{1}{l-m}}.
\end{equation}
It follows from the two elementary inequalities
\[
l(n-2-l) >m(n-2-m), \quad  (l+2)(n-4-l) >(m+2)(n-4-m),
\]
that
\begin{equation}
0< \underline r_1< \underline r_2 < \underline r_3.
\end{equation}
Next, we will give the properties of the functions  given in \eqref{OuterSubSolution}.
\begin{lemma}\label{LemmaPropertiesOfSubSol}
Let the functions $\underline u_{\rm out}, \underline v_{\rm out}, \underline w_{\rm out}$ be given as in \eqref{OuterSubSolution} and  $\underline r_1,\underline r_2, \underline r_3$ be defined in  \eqref{UnderlineR0} and \eqref{UnderlineR2}. There holds
\begin{equation}\label{PropertiesOfSubSol}
\begin{cases}
-\Delta \underline u_{\rm out} =\underline v_{\rm out} \quad& \text{for all} \quad r>0, \\
-\Delta \underline v_{\rm out} =\underline w_{\rm out} \quad& \text{for all} \quad r>0,\\
-\Delta \underline w_{\rm out} \leqslant \underline u_{\rm out}^p \quad& \text{for all} \quad r>\underline r_1.
\end{cases}
\end{equation}
\end{lemma}
\begin{proof}
A direct computation shows that the first and second equalities in \eqref{PropertiesOfSubSol} hold.  In order to prove the remaining estimate, we apply \eqref{BinomialIne} to get
\begin{align*}
\left( Lr^{-m}-br^{-l}\right)^p &= (Lr^{-m})^p \left( 1-\frac{b}{L}r^{m-l}\right)^p \\
&\geqslant L^p r^{-mp} \left( 1-p\frac{b}{L}r^{m-l}\right)\\
&=L^p r^{-m-6}- pL^{p-1}br^{-l-6} 
\end{align*}
for $ r>\underline r_1$, which yields that 
\begin{align*}
-\Delta \underline w_{\rm out}-  \underline u_{\rm out}^p &= m(m+2)(m+4)(n-2-m)(n-4-m)(n-6-m)Lr^{-m-6} \\
&\qquad - l(l+2)(l+4)(n-2-l)(n-4-l)(n-6-l)br^{-l-6} \\
&\qquad -(Lr^{-m}-br^{-l})^p \\
&= L^p r^{-m-6}- pL^{p-1}br^{-l-6} - \left( Lr^{-m}-br^{-l}\right)^p \leqslant 0,\end{align*}
here we have used the fact that $l=m+\lambda_3$ is the root of \eqref{CharacteristicsEq}.  
\end{proof}
We are now in position to construct a sub-solution to \eqref{LaneEmdenSystem}.
\begin{lemma}\label{LemmaSubSolution}
Let the functions $\underline u_{\rm out}, \underline v_{\rm out}, \underline w_{\rm out}$ be given as in \eqref{OuterSubSolution} and  $\underline r_1,\underline r_2, \underline r_3$ be defined in  \eqref{UnderlineR0}, \eqref{UnderlineR2}. Define the following functions, 
\begin{equation} 
\underline u(x) :=
\begin{cases}
0, \quad & 0\leqslant |x|<\underline r_1,\\
\underline u_{\rm out}(|x|), \quad & |x|\geqslant \underline r_1, 
\end{cases}
\end{equation}
\begin{equation} 
\underline v(x) :=
\begin{cases}
0, \quad & 0\leqslant |x|<\underline r_2,\\
\underline v_{\rm out}(|x|), \quad & |x|\geqslant \underline r_2,
\end{cases}
\end{equation}
\begin{equation} 
\underline w(x) :=
\begin{cases}
0, \quad & 0\leqslant |x|<\underline r_3,\\
\underline w_{\rm out}(|x|), \quad & |x|\geqslant \underline r_3. 
\end{cases}
\end{equation}
Then $\underline u, \underline v, \underline w$ are nonnegative and Lipschitz continuous on $\R^n$ and satisfy
\begin{equation}\label{SubSolution}
\left\{
\begin{split}
&-\Delta \underline u \leqslant \underline v, \\
&-\Delta \underline v \leqslant \underline w, \\
&-\Delta \underline w \leqslant \underline u^p
\end{split}
\right.
\end{equation}
in the distributional sense in $\R^n$.
\end{lemma}
\begin{proof}
From \eqref{OuterSubSolution}, it is not hard to see that $\underline u, \underline v, \underline w$ are nonnegative and Lipschitz continuous on $\R^n$.
It is sufficient to prove Lemma \ref{LemmaSubSolution} in 4 following cases: $0<|x|<\underline r_1$, $\underline r_1<|x|<\underline r_2$, $\underline r_2<|x|<\underline r_3$ and $\underline r_3 < |x|$.  If $0<|x|<\underline r_1$, \eqref{SubSolution} is trivial since  $\underline u = \underline v =\underline w =0$. We can also see that the case $\underline r_3 < |x|$ is a consequence of \eqref{LemmaPropertiesOfSubSol}. Next, we consider the others. 

When $\underline r_1<|x|<\underline r_2$,  we have $\underline u = \underline u_{\rm out}$ and $\underline v=\underline w=0$. Immediately, we deduce the second and third line of \eqref{SubSolution}. Besides, thanks to $\underline v_{\rm out}(r) \leq 0$ for $r<\underline r_1$, we obtain
\[
-\Delta \underline u =-\Delta \underline u_{\rm out} =\underline v_{\rm out} \leqslant 0 =\underline v.
\]
This establishes \eqref{SubSolution}. 

In the last case, $\underline r_2<|x|<\underline r_3$, there holds $\underline u = \underline u_{out}$, $\underline v = \underline v_{out}$ and $\underline w=0$. The third inequality of \eqref{SubSolution} can be easily seen and the first one holds by the first line of \eqref{PropertiesOfSubSol}. Meanwhile, we have
\[
-\Delta \underline v = -\Delta \underline v_{\rm out} = \underline w_{\rm out} \leqslant 0 \leqslant \underline w,
\] 
i.e. the second line of \eqref{SubSolution} holds. Combining the previous 4 cases, we conclude Lemma \ref{LemmaSubSolution} and give the explicit formula of the entire sub-solution of \eqref{LaneEmdenSystem} in the supercritical case.

\end{proof}
\subsection{Construction of a super-solution \eqref{LaneEmdenSystem}}
Set \[k_0:= \min\{m+\lambda_4,2l-m\}.\] 
Clearly $k_0>l$.  Before constructing a super-solution to \eqref{LaneEmdenSystem}, it is worth noting that there is a fundamental difference between the construction of entire super- and sub-solutions to \eqref{LaneEmdenSystem}. To be more precise, in the inner region around the origin, we replace the constant zero being part of the sub-solution to some function of the form   $A(r^2+\epsilon)^{-\kappa}$, where $A,\kappa$ are relevant positive constants and  $0<\epsilon \ll 1$ to be specified.  Besides, far away from the origin, we look for an outer super-solution which is similar \eqref{OuterSubSolution}, however, with a minor modification. To make it clear, we will state the inner and outer parts of the super-solution in the next two lemmas.
\begin{lemma}\label{LemmaInnerSuperSolution}
For any $\epsilon >0$, we consider the following functions
\begin{equation}\label{InnerSuperSolution}
\left\{
\begin{split}
\overline u_{\rm in,\epsilon} &:=L(r^2+\epsilon)^{-\frac{m}{2}}, \\
\overline v_{\rm in,\epsilon} &:=m(n-2-m) L(r^2+\epsilon)^{-\frac{m+2}{2}}, \\
\overline w_{\rm in,\epsilon} &:=m(m+2)(n-2-m)(n-4-m) L(r^2+\epsilon)^{-\frac{m+4}{2}},
\end{split}
\right.
\end{equation}
which satisfy
\begin{equation}
\begin{cases}
-\Delta \overline u_{\rm in,\epsilon} \geqslant  \overline v_{\rm in,\epsilon},  \\
-\Delta \overline v_{\rm in,\epsilon} \geqslant \overline w_{\rm in,\epsilon}, \\
-\Delta \overline w_{\rm in,\epsilon} \geqslant \overline u_{\rm in,\epsilon}^p
\end{cases}
\end{equation}
for all $ r>0$.
\end{lemma}
\begin{proof}
We recall the following estimate from \cite[Lemma 3.6]{Win10},
\[
-\Delta (r^2+\epsilon)^{-\frac{k}{2}} \geqslant k(n-2-k)(r^2+\epsilon)^{-\frac{k+2}{2}} \quad \text{for all } r\geqslant 0.
\]
We substitute $k$ by $m, m+2$ and $m+4$, repectively to obtain that
\begin{align*}
-\Delta  \overline u_{\rm in,\epsilon} &\geqslant m(n-2-m)L(r^2+\epsilon)^{-\frac{m+2}{2}} = \overline v_{\rm in,\epsilon},\\
-\Delta  \overline v_{\rm in,\epsilon} &\geqslant m(m+2)(n-2-m)(n-4-m)L(r^2+\epsilon)^{-\frac{m+4}{2}} = \overline w_{\rm in,\epsilon},
\end{align*}
and
\begin{align*}
-\Delta  \overline w_{\rm in,\epsilon} &\geqslant  m(m+2)(m+4)(n-2-m)(n-4-m)(n-6-m) L(r^2+\epsilon)^{-\frac{m+6}{2}}\\
&  = L^p (r^2+\epsilon)^{-\frac{mp}{2}} = \overline u_{\rm in,\epsilon}^p. 
\end{align*}
Proof of lemma  is complete.
\end{proof}

\begin{lemma}\label{LemmaOuterSuperSolution}
Let $p>\pc$. For all $k\in(l,k_0)$, there exists $c>0$ such that the  following functions  $\overline u_{\rm out}, \overline v_{\rm out}, \overline w_{\rm out}$ given by
\begin{equation}\label{OuterSupSolCase1}
\left\{
\begin{split}
\overline u_{\rm out}(r) &= Lr^{-m}-br^{-l} +cr^{-k}, \\
\overline v_{\rm out}(r) &= m(n-2-m)Lr^{-m-2}- l(n-2-l)br^{-l-2}\\
& \qquad +k(n-2-k)cr^{-k-2},\\
\overline w_{\rm out}(r) &= m(m+2)(n-2-m)(n-4-m)Lr^{-m-4} \\
&\qquad - l(l+2)(n-2-l)(n-4-l)br^{-l-4} \\
&\qquad+ k(k+2)(n-2-k)(n-4-k)c r^{-k-4}
\end{split}
\right.
\end{equation}
satisfy the following estimates
\begin{equation}\label{OuterSuperSolution}
\begin{cases}
-\Delta \overline u_{\rm out} =  \overline v_{out} \quad& \text{for all} \quad r>0, \\
-\Delta \overline v_{\rm out} = \overline w_{out} \quad& \text{for all} \quad r>0, \\
-\Delta \overline w_{\rm out} \geqslant \overline u_{out}^p \quad& \text{for all} \quad r>\overline r_1,
\end{cases}
\end{equation}
where 
\[
\overline r_1 := \left( \frac{c}{b}\right)^{\frac{1}{k-l}}.
\]
\end{lemma}
\begin{proof}
Let $k \in (l,k_0)$ be given, clearly
\[
P(k) = k(k+2)(k+4)(n-2-k)(n-4-k)(n-6-k) -pL^{p-1} >0.
\]
To fulfill \eqref{OuterSuperSolution}, the constant $c$ is required to satisfy 
\[
cP(k) \geq C_p L^{p-2}b^2 \overline r_1^{k+m-2l} \quad \text{and} \quad \left( \frac{c}{b}\right)^{\frac{1}{k-l}}\geq \left( \frac{b}{L}\right)^{\frac{1}{l-m}},
\]
where $C_p$ is the constant appearing in Lemma \ref{LemmaBinomialIne}. Indeed, going back to \eqref{UnderlineR0} to get that $\overline r_1\geq \underline r_1$ and that
\[
\overline u_{\rm out}(r) > Lr^{-m}-br^{-l}\geqslant 0
\]
for all $r\geq \overline r_1$. The direct computation shows the first and second equalities of \eqref{OuterSuperSolution}, while 
\begin{align*}
-\Delta \overline w_{\rm out} - \overline u_{\rm out}^p &= m(m+2)(m+4)(n-2-m)(n-4-m)(n-6-m)Lr^{-m-6} \\
&\qquad - l(l+2)(l+4)(n-2-l)(n-4-l)(n-6-l)br^{-l-6} \\
&\qquad + k(k+2)(k+4)(n-2-k)(n-4-k)(n-6-k)br^{-k-6} \\
&\qquad -(Lr^{-m}-br^{-l}+cr^{-k})^p.
\end{align*}
From the  definition of $\overline r_1$, we notice that
\[
0< z:= \frac{b}{L}r^{m-l}-\frac{c}{L}r^{m-k} \leqslant 1
\]
 for $r\geqslant \overline r_1$. Hence, thanks to Lemma \ref{LemmaBinomialIne}, we give the estimate
\begin{align*}
& (Lr^{-m}-br^{-l}+cr^{-k})^p= L^p r^{-mp} \left(  1-\left(  \frac{b}{L}r^{m-l}-\frac{c}{L}r^{m-k} \right)   \right)^p \\
&\qquad \leqslant L^p r^{-mp} \left(  1- p \left(  \frac{b}{L}r^{m-l}-\frac{c}{L}r^{m-k} \right)    + C_p \left(  \frac{b}{L}r^{m-l}-\frac{c}{L}r^{m-k} \right)^2 \right) \\
&\qquad \leqslant L^p r^{-mp}  -pL^{p-1}br^{-l-6} + pL^{p-1}cr^{-k-6} +C_p L^{p-2}b^2 r^{-6+m-2l}
\end{align*}
for $r\geqslant \overline r_1$. Hence,
\begin{align*}
-\Delta \overline w_{\rm out} - \overline u_{\rm out}^p  &\geq r^{-k-6}\left( c P(k) - C_p L^{p-2}b^2 r^{k+m-2l}\right) \\
&\geqslant r^{-k-6}\left( c P(k) - C_p  L^{p-2}b^2 \overline r_1^{k+m-2l} \right) \geqslant 0
\end{align*}
for $r\geqslant \overline r_1$. 
\end{proof}
Our next task is to give the borderline between inner and outer super-solutions introduced in two preceding lemmas. 
\begin{lemma}\label{LemmaBorderline}
Let 
$k\in(l,k_0)$ be given. Then, there exists $\epsilon_0>0$ such that  for any $\epsilon  \in(0,\epsilon_0)$, the following quantities
\begin{equation}\label{BorderlineR1}
\overline r_1(\epsilon) := \sup\{r>0 : \overline u_{\rm in,\epsilon} < \overline u_{\rm out} \text{ in } (0,r)\} \in (0,+\infty],
\end{equation}
\begin{equation}\label{BorderlineR2}
\overline r_2(\epsilon) := \sup\{r>0 : \overline v_{\rm in,\epsilon} < \overline v_{\rm out} \text{ in } (0,r)\} \in (0,+\infty],
\end{equation}
and
\begin{equation}\label{BorderlineR3}
\overline r_3(\epsilon) := \sup\{r>0 : \overline w_{\rm in,\epsilon} < \overline w_{\rm out} \text{ in } (0,r)\} \in (0,+\infty],
\end{equation}
fulfill the following estimates
\begin{equation}\label{Ordering}
\overline r_1 < \overline r_1(\epsilon) < \overline r_2 <\overline r_2(\epsilon) <\overline r_3 <\overline r_3(\epsilon) < \overline r_3 +1.
\end{equation}
Here,    $\overline u_{\rm in,\epsilon}, \overline v_{\rm in,\epsilon}, \overline w_{\rm in,\epsilon}, \overline u_{\rm out}, \overline v_{\rm out}, \overline w_{\rm out}$ are given  in \eqref{InnerSuperSolution} and \eqref{OuterSupSolCase1} and
\[
\overline r_1 := \left( \frac{c}{b}\right)^{\frac{1}{k-l}}, \quad \overline r_2 :=\left( \frac{k(n-2-k)c}{l(n-2-l)b}\right)^{\frac{1}{k-l}},
\]
and
\[
\overline r_3 :=\left( \frac{k(k+2)(n-2-k)(n-4-k)c}{l(l+2)(n-2-l)(n-4-l)b}\right)^{\frac{1}{k-l}}.
\]
\end{lemma}
\begin{proof}
To overcome this lemma, let us present the gap functions between inner and outer parts of the super-solutions $(\overline u_{\epsilon}, \overline v_{\epsilon}, \overline w_{\epsilon})$
\begin{align*}
\psi_{1,\epsilon}(r) := \overline u_{\rm in,\epsilon}(r) - \overline u_{\rm out}(r),\\
\psi_{2,\epsilon}(r) := \overline v_{\rm in,\epsilon}(r) - \overline v_{\rm out}(r),
\end{align*}
and 
\[
\psi_{3,\epsilon}(r) := \overline w_{\rm in,\epsilon}(r) - \overline w_{\rm out}(r).
\]
In the next stage, we study the pointwise convergence of $\psi_{1,\epsilon},\psi_{2,\epsilon}$ and $\psi_{3,\epsilon}$ as  $\epsilon \searrow 0$. It is not hard to see that 
\begin{align*}\label{ConvergencePsiEpsilon}
&\psi_{1,\epsilon}(r) \nearrow \psi_1(r) :=  br^{-l-4}- cr^{-k-4},\\
&\psi_{2,\epsilon}(r) \nearrow \psi_2(r) :=  l(n-2-l)br^{-l-4}- k(n-2-k)cr^{-k-4},\\
&\psi_{3,\epsilon}(r) \nearrow \psi_3(r) := l(l+2)(n-2-l)(n-4-l) br^{-l-4}\\
&\qquad \qquad- k(k+2)(n-2-k)(n-4-k)cr^{-k-4} 
\end{align*}
for all $r>0$ in $C_{loc}^1((R,\infty))$ sense. In addition, we also notice the sign of limiting functions,
\begin{equation}\label{Psi1}
 \psi_1
\begin{cases}
<0 \quad& \text{for all } r<\overline r_1,\\
=0 \quad& \text{for } r=\overline r_1,\\
>0 \quad& \text{for all } r>\overline r_1.
\end{cases}
\end{equation}
\begin{equation}\label{Psi2} 
\psi_2
\begin{cases}
<0 \quad& \text{for all } r<\overline r_2,\\
=0 \quad& \text{for } r=\overline r_2,\\
>0 \quad& \text{for all } r>\overline r_2,
\end{cases}
\end{equation}
and 
\begin{equation}\label{Psi3} 
\psi_3
\begin{cases}
<0 \quad& \text{for all } r<\overline r_3,\\
=0 \quad& \text{for } r=\overline r_3,\\
>0 \quad& \text{for all } r>\overline r_3.
\end{cases}
\end{equation}
Combining  with  the monotone convergence of $\psi_{i,\epsilon} (i=1,2,3)$, we obtain that each function has exactly one zero point $\overline r_i(\epsilon) > \overline r_i$
 for all $\epsilon >0$ . Hence, for all $\eta >0$,  there exists $\epsilon_i(\eta)>0$ such that 
\[
\psi_{i,\epsilon}(\overline r_i +\eta) >0
\]
 when $\epsilon \in (0,\epsilon_i(\eta))$. Therefore, 
\[
\overline r_i < \overline r_i(\epsilon) < \overline r_i +\eta 
\] 
for all  $\epsilon \in (0,\epsilon_i(\eta))$. For $i=1,2,3$, one can choose $\eta$ is equal to $(\overline r_2 -\overline r_1)/2$, $(\overline r_3 -\overline r_2)/2$ and $1$,  respectively to complete the order \eqref{Ordering}. 
\end{proof}
Now we are in position to construct a super-solution $(\overline u_{\epsilon}, \overline v_{\epsilon}, \overline w_{\epsilon})$ of \eqref{LaneEmdenSystem}. 
\begin{lemma}\label{LemmaSuperSolution}
Let $p>\pc$ and $ b>0$,  then there exist $k>l,c>0$ and $\epsilon_0 >0$ such that whenever $\epsilon \in (0,\epsilon_0)$, the functions
\begin{equation} 
\overline u_{\epsilon}(x) := 
\begin{cases}
\overline u_{\rm in,\epsilon}(|x|), \quad & 0\leqslant |x|<\overline r_1(\epsilon),\\
\overline u_{\rm out}(|x|), \quad & |x|\geqslant \overline r_1(\epsilon),
\end{cases}
\end{equation}
\begin{equation} 
\overline v_{\epsilon}(x) := 
\begin{cases}
\overline v_{\rm in,\epsilon}(|x|), \quad & 0\leqslant |x|<\overline r_2(\epsilon),\\
\overline v_{\rm out}(|x|), \quad & |x|\geqslant \overline r_2(\epsilon). 
\end{cases}
\end{equation}
\begin{equation} 
\overline w_{\epsilon}(x) := 
\begin{cases}
\overline w_{\rm in,\epsilon}(|x|), \quad & 0\leqslant |x|<\overline r_3(\epsilon),\\
\overline w_{\rm out}(|x|), \quad & |x|\geqslant \overline r_3(\epsilon), 
\end{cases}
\end{equation}
are nonnegative, Lipschitz continuous on $\R^n$ and satisfy 
\begin{equation}\label{SuperSolutionCase1}
\begin{cases}
-\Delta \overline u_{\epsilon} \geqslant \overline v_{\epsilon},  \\
-\Delta \overline v_{\epsilon} \geqslant \overline w_{\epsilon}, \\
-\Delta \overline w_{\epsilon} \geqslant \overline u_{\epsilon}^p,
\end{cases}
\end{equation}
in the distributional sense on $\R^n$. Here, $ \overline r_1(\epsilon), \overline r_2(\epsilon), \overline r_3(\epsilon)$ are given in \eqref{BorderlineR1}, \eqref{BorderlineR2} and \eqref{BorderlineR3} and the functions $\overline u_{\rm in,\epsilon}, \overline v_{\rm in,\epsilon}, \overline w_{\rm in,\epsilon}, \overline u_{\rm out}, \overline v_{\rm out}, \overline w_{\rm out}$ are defined through \eqref{InnerSuperSolution} and \eqref{OuterSupSolCase1}.
\end{lemma}
\begin{proof}
 There is a need  to prove Lemma \ref{LemmaSuperSolution} in four sub-domains $0<|x|<\overline r_1(\epsilon), \overline r_1(\epsilon)<|x|<\overline r_2(\epsilon), \overline r_2(\epsilon)<|x|<\overline r_3(\epsilon)$ and $\overline r_3(\epsilon)<|x|$. The case $0<|x|<\overline r_1(\epsilon)$ and  $\overline r_3(\epsilon)<|x|$ are directly followed by Lemma  \ref{LemmaInnerSuperSolution} and \ref{LemmaOuterSuperSolution}. Now, we deal with the two remaining cases. 

When $\overline r_1(\epsilon)<|x|<\overline r_2(\epsilon)$, there holds 
\[ 
\overline u_{\epsilon}(x) = \overline u_{\rm out}(|x|), \quad \overline v_{\epsilon}(x) = \overline v_{\rm in,\epsilon}(|x|), \quad w_{\epsilon}(x) = \overline w_{\rm in,\epsilon}(|x|).
\]
Hence, using Lemma \ref{LemmaOuterSuperSolution} again, the first inequality of \eqref{SuperSolutionCase1} can be obtained, since
\[
-\Delta \overline u_{\epsilon} = -\Delta \overline u_{\rm out} \geqslant \overline v_{\rm out} \geq \overline u_{\rm in,\epsilon} = \overline u_{\epsilon}.
\] 
The second and third one are consequences of Lemma \ref{LemmaInnerSuperSolution}. Indeed,
\[
-\Delta \overline v_{\epsilon} = -\Delta \overline v_{\rm in,\epsilon} \geqslant \overline w_{\rm in,\epsilon} = \overline w_{\epsilon}
\] 
and 
\[
-\Delta \overline w_{\epsilon} = -\Delta \overline w_{\rm in,\epsilon} \geqslant \overline u_{\rm in,\epsilon}^p  \geqslant \overline u_{\rm out}^p = \overline u_{\epsilon}^p,
\] 
here we have used the definition of $\overline r_1(\epsilon)$ and $\overline r_2(\epsilon)$. From this, we  can conclude \eqref{SuperSolutionCase1} in the region $\overline r_1(\epsilon)<|x|<\overline r_2(\epsilon)$. 

In the region $\overline r_2(\epsilon)<|x|<\overline r_3(\epsilon)$, we can argue similarly keeping in mind that $\overline v_{\epsilon}(x) = \overline v_{\rm out}(|x|)$ in this scenario. Thus, proof of Lemma \ref{LemmaSuperSolution} is complete.
\end{proof}

\section{The critical case $p=\pc$}
This section is devoted to construction of sub- and super-solutions to \eqref{LaneEmdenSystem} in the critical case $p=\pc$. As shown in the statement of Theorem \ref{MainThmCriticalCase}, the asymptotic behavior of solutions of \eqref{triharmonic} involves a logarithmic correction leading to more complicated proof of the borderline between sub- and super-solutions. As always, we divide our construction into two subsections.

\subsection{Construction of a sub-solution to \eqref{LaneEmdenSystem}}
Following the way of construction of the sub-solution in the supercritical case presented in subsection \ref{subsec-supercritical-sub}, we define the following functions
\begin{equation}\label{OuterSubSolCase2}
\left\{
\begin{split}
\underline u_{\rm out}(r) &:= Lr^{-m} - br^{-l} \log \frac{r}{R},\\
\underline v_{\rm out}(r) &:= m(n-2-m)Lr^{-m-2} +
\left[ \begin{split}
- l(n-2-l) \log \frac{r}{R} + (n-2-2l)
\end{split} \right]br^{-l-2},\\
\underline w_{\rm out}(r) &:= m(m+2)(n-2-m)(n-4-m)Lr^{-m-4}\\
&\qquad \qquad - l(l+2)(n-2-l)(n-4-l)br^{-l-4} \log \frac{r}{R} \\
&\qquad \qquad + (n-2-2l)(l+2)(n-4-l) b r^{-l-4},
\end{split}
\right.
\end{equation}
where $R$ is positive constant to be chosen later. Next, we study the sign of $\underline u_{\rm out}, \underline v_{\rm out}$ and  $\underline w_{\rm out}$.
\begin{lemma}\label{LemmaComparisionUout}
Let $\underline r_1, \underline r_2, \underline r_3$ be defined as follows
 \begin{equation}\label{UnderlineRcase2}
\begin{cases}
\underline r_1 := \sup\{ r>0 : \underline u_{\rm out}(r) \leqslant 0 \},\\
\underline r_2 := \sup\{ r>0 : \underline v_{\rm out}(r) \leqslant 0 \},\\
\underline r_3 := \sup\{ r>0 : \underline w_{\rm out}(r) \leqslant 0 \}.
\end{cases}
\end{equation}
There exists large enough $R>0$ such that the numbers $\underline r_1, \underline r_2, \underline r_3$ are well-defined and ordered as follows
\begin{equation}\label{OrderRunder}
R<\underline r_1 < \underline r_2 < \underline r_3 <+\infty.
\end{equation}
\end{lemma}
\begin{proof}
Let real number $\mu$ be given such that 
\begin{align*}
\mu &> \frac{1}{l-m},\\
\mu &> \frac{n-2-2l}{(n-2-l-m)(l-m)},\\
\mu &> \frac{l(n-2-l)}{m(n-2-m)(l-m)},\\
\mu &> \frac{n-2-2l}{l(n-2-l)},
\end{align*}
and that 
\[
\mu > \frac{l(l+2)(n-2-l)(n-4-l)}{m(m+2)(n-2-m)(n-4-m)(l-m)}.
\]
Let $R,r_1$ be chosen as follows
\[
R:= e^{-\mu} \left( \frac{\mu b}{L} \right)^{\frac{1}{l-m}}, \quad r_1:= \left( \frac{\mu b}{L} \right)^{\frac{1}{l-m}}.
\]
Now,  in terms of the outer part of $(\underline u, \underline v, \underline w)$, we define
\[
\left\{
\begin{split}
\phi_1(r) &= r^l \underline u_{\rm out}(r) \\
&= Lr^{l-m} -b\log \frac{r}{R},\\
\phi_2(r) &= r^{l+2} \underline v_{\rm out}(r) \\
&= m(n-2-m)Lr^{l-m} - l(n-2-l)b\log \frac{r}{R} + (n-2-2l)b,\\
\phi_3(r) &= r^{l+4} \underline w_{\rm out}(r)\\
&= m(m+2)(n-2-m)(n-4-m)Lr^{l-m}\\
&\quad - l(l+2)(n-2-l)(n-4-l)b \log \frac{r}{R} + (n-2-2l)(l+2)(n-4-l) b .
\end{split}
\right.
\]
It is obvious that $\phi_1(r_1) =0$ and for $r > r_1$, there holds 
\begin{align*}
\frac{d}{dr} \phi_1(r) = \frac{L}{r} \Big[ (l-m) r^{l-m} - \frac{b}{L} \Big] > \frac{b}{r} \Big[ (l-m) \mu -  1\Big] >0,
\end{align*}
which tells us that $\phi_1$ only has a positive root $r_1$. Hence $\underline r_1 =r_1$. AFor the function $\phi_2$, at $r_1$, we calculate to see that
\begin{align*}
\phi_2(r_1) &= m(n-2-m)Lr_1^{l-m} - l(n-2-l)b\log \frac{r_1}{R} + (n-2-2l)b\\
&= m(n-2-m)b\log \frac{r_1}{R}- l(n-2-l)b\log \frac{r_1}{R} + (n-2-2l)b\\
&=\Big[ - (l-m)(n-2-l-m) \mu + (n-2-2l ) \Big] b <0,
\end{align*}
and that
\begin{align*}
\frac{d}{dr} \phi_2(r) &= \frac{L}{r} \Big[ (l-m) m(n-2-m) r^{l-m} - l(n-2-l)\frac{b}{L} \Big] \\
&> \frac{b}{r} \Big[  m(n-2-m) (l-m) \mu -  l(n-2-l) \Big] >0
\end{align*}
for all $ r >\underline  r_1$. Thus, there is only one zero point $r=\underline r_2 > \underline r_1$ of $\phi_2$. This fact   implies that 
\[
m(n-2-m)L \underline r_2^{l-m} = l(n-2-l)b\log \frac{\underline r_2}{R} - (n-2-2l)b.
\]
Making a substitution into $\phi_3(r_2)$ to get that
\begin{align*}
\phi_3(\underline r_2)  &= m(m+2)(n-2-m)(n-4-m)L \underline r_2^{l-m}\\
&\qquad  - l(l+2)(n-2-l)(n-4-l) \log \frac{\underline r_2}{R} + (n-2-2l)(l+2)(n-4-l) b \\
&= (m+2)(n-4-m) \left[ l(n-2-l) b \log \frac{\underline r_2}{R} -(n-2-2l)b \right]\\
&\qquad - l(l+2)(n-2-l)(n-4-l) \log \frac{\underline r_2}{R} + (n-2-2l)(l+2)(n-4-l)b \\
&= -(l-m)(n-2-m-l)b \left[  l(n-2-l)  \log \frac{\underline r_2}{R} -(n-2-2l) \right]\\
&<0,
\end{align*}
because
\[
  \log \frac{\underline r_2}{R} >  \log \frac{ \underline r_1}{R}   = \mu > \frac{n-2-2l}{l(n-2-l)} .
\]
Meanwhile, for $r\geq \underline r_2 (> \underline r_1)$,
\begin{align*}
\phi_3'(r) &= m(m+2)(n-2-m)(n-4-m)(l-m)Lr^{l-m-1} \\
&\qquad -  l(l+2)(n-2-l)(n-4-l) \frac{b}{r}\\
 &= \frac{m(m+2)(n-2-m)(n-4-m)(l-m)b}{r} \\
&\quad\quad\times \left( \frac{L}{b}r^{l-m} - \frac{l(l+2)(n-2-l)(n-4-l)}{m(m+2)(n-2-m)(n-4-m)(l-m)} \right)\\
&> \frac{m(m+2)(n-2-m)(n-4-m)(l-m)b}{r} \\
&\quad\quad\times \left( \mu -  \frac{l(l+2)(n-2-l)(n-4-l)}{m(m+2)(n-2-m)(n-4-m)(l-m)}  \right)\\
&>0.
\end{align*}
Combining these above facts to obtain that $\phi_3$ admits only one positive root $\underline r_3 >\underline  r_2$. That leads to the well definition of $\underline r_3$ and \eqref{OrderRunder}.
\end{proof}
The following result is similar to  Lemma \ref{LemmaPropertiesOfSubSol} hower for the critical case.
\begin{lemma}\label{LemmaPropOfSubSolCase2}
Let $\underline u_{\rm out}, \underline v_{\rm out}, \underline w_{\rm out}$ be defined as in \eqref{OuterSubSolCase2} and $\underline r_1$ be stated as in Lemma \ref{LemmaComparisionUout}, there holds
\begin{equation}\label{PropOfSubSolCase2}
\begin{cases}
-\Delta \underline u_{\rm out} = \underline v_{\rm out} &\quad \text{for all } r>0,\\
-\Delta \underline v_{\rm out} = \underline w_{\rm out} &\quad \text{for all }  r>0,\\
-\Delta \underline w_{\rm out} \leqslant \underline u_{\rm out}^p  &\quad \text{for all } r>\underline r_1.
\end{cases}
\end{equation}
\end{lemma}
\begin{proof}
The first and second lines of \eqref{PropOfSubSolCase2} can be easily verified by computing directly from the definition of $\underline u_{\rm out}$ and $\underline v_{\rm out}$. To obtain the third line of \eqref{PropOfSubSolCase2}, we make us of the following identity
\begin{align}\label{identity}
&\Delta \left( r^{-\alpha} \left(\log \frac{r}{R}\right)^{\gamma} \right) = -\alpha(n-2-\alpha)r^{-\alpha-2}\left(\log \frac{r}{R} \right)^{\gamma}\notag\\
& + \gamma(n-2-2\alpha) r^{-\alpha-2}\left(\log \frac{r}{R} \right)^{\gamma-1} +\gamma(\gamma-1) r^{-\alpha-2}\left(\log \frac{r}{R} \right)^{\gamma-2}.
\end{align}
 with   $\gamma=1$ and $\alpha = l+4$. Hence, we obtain
\[ \begin{split}
&-\Delta \underline w_{\rm out} - \underline u_{\rm out}^p\\
&= m(m+2)(m+4)(n-2-m)(n-4-m)(n-6-m)Lr^{-m-6} \\
&\qquad + l(l+2)(n-2-l)(n-4-l) 
\left[ \begin{split}
 &  -(l+4)(n-6-l) \log \frac{r}{R}\\
&+(n-10-2l)  \\
\end{split} \right]br^{-l-6} \\
&\qquad +(n-2-2l)(l+2)(n-4-l)(l+4)(n-6-l)br^{-l-6}\\
&\qquad -\left(Lr^{-m}-br^{-l}\log \frac{r}{R}\right)^p\\
&=m(m+2)(m+4)(n-2-m)(n-4-m)(n-6-m)Lr^{-m-6}\\
&\qquad -l(l+2)(l+4)(n-2-l)(n-4-l)(n-6-l)br^{-l-6}\log \frac{r}{R}\\
&\qquad -\left(Lr^{-m}-br^{-l}\log \frac{r}{R}\right)^p,
\end{split}
\]
here we notice  that 
\[
l(n-2-l)(n-10-2l) +(l+4)(n-6-l)(n-2-2l)=0.
\]
In order to apply Lemma \ref{LemmaBinomialIne}, we need to verify that 
\[ 
0\leqslant \underline u_{\rm out}(r) < Lr^{-m} 
\]
for all   $r \geqslant \underline r_1$. Indeed, 
\[
\underline u_{\rm out}(r)-Lr^{-m} = b\log \frac{r}{R} >0 
\]
for all $r>R$. Hence, 
\begin{align*}
-\Delta \underline w_{\rm out} - \underline u_{\rm out}^p= L^p r^{-m-6}- pL^{p-1}br^{-l-6}\log \frac{r}{R} - \left( Lr^{-m}-br^{-l}\log \frac{r}{R} \right)^p \leqslant  0.
\end{align*}
 Thus, \eqref{PropOfSubSolCase2} is proved.
\end{proof}
Simply repeating the way in the supercritical case, we complete the construction of the entire sub-solution of \eqref{LaneEmdenSystem}. The proof  is similar to Lemma \ref{LemmaSubSolution} and we omit the detail  here.
\begin{lemma}\label{LemmaSubSolCase2}
Let $\underline r_1, \underline r_2, \underline r_3$ and $\underline u_{\rm out}, \underline v_{\rm out}, \underline w_{\rm out}$ be defined as  in \eqref{UnderlineRcase2} and \eqref{OuterSubSolCase2}. We consider the following functions
\[
\underline u(x) :=
\begin{cases}
0, \quad & 0\leqslant |x|<\underline r_1,\\
\underline u_{\rm out}(|x|), \quad & |x|\geqslant \underline r_1, 
\end{cases}
\]
\[
\underline v(x) :=
\begin{cases}
0, \quad & 0\leqslant |x|<\underline r_2,\\
\underline v_{\rm out}(|x|), \quad & |x|\geqslant \underline r_2,
\end{cases}
\]
\[
\underline w(x) :=
\begin{cases}
0, \quad & 0\leqslant |x|<\underline r_3,\\
\underline w_{\rm out}(|x|), \quad & |x|\geqslant \underline r_3. 
\end{cases}
\]
Then $\underline u, \underline v, \underline w$ are nonnegative, Lipschitz continuous on $\R^n$ and such that 
\begin{equation}\label{SubSolutionCase2}
\begin{cases}
-\Delta \underline u \leqslant \underline v, \\
-\Delta \underline v \leqslant \underline w, \\
-\Delta \underline w \leqslant \underline u^p
\end{cases}
\end{equation}
in the distributional sense in $\R^n$.
\end{lemma}

\subsection{Construction of a super-solution to \eqref{LaneEmdenSystem}}
We divide $\R^n$ into two regions. The inner super-solutions are still remained as stated in Lemma \ref{LemmaInnerSuperSolution}. Next, we will show the explicit formula of outer super-solution as follows, given $\beta \in (0,1)$,
\begin{equation}\label{FormulaOuterSuperSolCase2}
\left\{
\begin{split}
\overline u_{\rm out}(r) &:= Lr^{-m} -b r^{-l} \log \frac{r}{R} + cr^{-l} \left( \log \frac{r}{R} \right)^{\beta},\\
\overline v_{\rm out}(r) &:= m(n-2-m)Lr^{-m-2} - \Big[ l(n-2-l) \log \frac{r}{R} - (n-2-2l)\Big] br^{-l-2} \\
& \quad +  \left[
\begin{split}
 &l(n-2-l) \left( \log \frac{r}{R} \right)^{\beta}\\
& -\beta(n-2-2l)  \left( \log \frac{r}{R} \right)^{\beta-1}\\
&+\beta(1-\beta)  \left( \log \frac{r}{R} \right)^{\beta-2}  
\end{split}
\right]cr^{-l-2}\\
\overline w_{\rm out}(r) &:=  m(m+2)(n-2-m)(n-4-m)Lr^{-m-4} \\
& \quad + (l+2)(n-4-l) \Big[ l(n-2-l) \log \frac{r}{R} -(n-2-2l) \Big]br^{-l-4}\\
& \quad +\left[
\begin{split}
& l (l+2) (n-2-l) (n-4-l) \left( \log \frac{r}{R} \right)^{\beta} \\
& - \beta(n-2-2l) (l+2)(n-4-l) \left( \log \frac{r}{R} \right)^{\beta-1}  \\
&   +\beta(1-\beta) (l(n-2-l)+(l+2)(n-4-l)) \left( \log \frac{r}{R} \right)^{\beta-2}  \\
&+ \beta(1-\beta)(2-\beta) (n-2-2l) \left( \log \frac{r}{R} \right)^{\beta-3} \\
& - \beta(1-\beta)(2-\beta)(3-\beta) \left( \log \frac{r}{R} \right)^{\beta-4} 
\end{split} 
\right]cr^{-l-4}
\end{split}
\right.
\end{equation}

\begin{lemma}\label{LemOuterSupSol}
Let $\overline u_{\rm out}, \overline v_{\rm out}, \overline w_{\rm out}$ be defined as in \eqref{FormulaOuterSuperSolCase2}. There exists $c_1>0$ such that for all $c>c_1$, there holds
\begin{equation}\label{OuterSuperSolutionCase2}
\begin{cases}
-\Delta \overline u_{\rm out} =  \overline v_{out} \quad& \text{for all} \quad r>R, \\
-\Delta \overline v_{\rm out} = \overline w_{out} \quad& \text{for all} \quad r>R, \\
-\Delta \overline w_{\rm out} \geqslant \overline u_{\rm out}^p \quad& \text{for all} \quad r>\overline r_1,
\end{cases}
\end{equation}
where 
\[
\overline r_1 : =R \exp\left( \left(\frac{c}{b}\right)^{\frac{1}{1-\beta}}\right).
\]
\end{lemma}
\begin{proof}
Let $C_p$ be the constant appearing in Lemma \ref{LemmaBinomialIne} and $c_1$ be large enough such that 
\[
\left\{
\begin{split}
&(l-m) \left( \frac{c_1}{b} \right)^{\frac{1}{1-\beta}} > 4-\beta,\\
&(l-m)L \left( R \exp\left( \left(\frac{c_1}{b}\right)^{\frac{1}{1-\beta}}\right) \right)^{l-m} >b,\\
&\frac{c_1}{b} >  \left( \frac{24(2-\beta)(3-\beta)(3n^2-12n+44)}{3n^4-24n^3+72n^2-96n+304} \right)^{\frac{1-\beta}{2}},
\end{split}
\right. 
\]
as well as
\[
L  \left( R \exp\left( \left(\frac{c}{b}\right)^{\frac{1}{1-\beta}}\right) \right)^{l-m} > b \left( \frac{c}{b} \right)^{\frac{1}{1-\beta}}, 
\]
and
\begin{align*}
c \left( R \exp\left( \left(\frac{c}{b}\right)^{\frac{1}{1-\beta}}\right) \right)^{l-m} \left( \frac{c}{b} \right)^{\frac{\beta-4}{1-\beta}}
 > \frac{16C_p L^{p-2}b^2}{\beta(1-\beta)(3n^4-24n^3+72n^2-96n+304)}
\end{align*}
for all $c >c_1$. 
Thanks to \eqref{identity}, the first two equality of \eqref{OuterSuperSolutionCase2} can be checked directly. Thus, it suffices  to prove the third inequality of \eqref{OuterSuperSolutionCase2}, that is $-\Delta \overline w_{\rm out} \geqslant \overline u_{\rm out}^p$ for all $r > \overline r_1$. First, we compute $-\Delta \overline w_{\rm out}$. Using \eqref{identity} again, we deduce that
\begin{align*}
&-\Delta \overline w_{\rm out}(r) =  m(m+2)(m+4)(n-2-m)(n-4-m)(n-6-m)Lr^{-m-6} \\
&\qquad- l(l+2)(l+4)(n-2-l)(n-4-l)(n-6-l)br^{-l-6}  \log \frac{r}{R} \\
& \qquad + (l+2)(n-4-l)  
\left[ \begin{aligned}
 &l(n-2-l)(n-10-2l) \\
& +(l+4)(n-6-l)(n-2-2l)
\end{aligned} \right]
 b r^{-l-6}\\
& \qquad+ l(l+2)(n-2-l)(n-4-l)
\left[ \begin{aligned}
 & (l+4)(n-6-l) \left(  \log \frac{r}{R} \right)^{\beta} \\
&-\beta(n-10-2l) \left( \log \frac{r}{R} \right)^{\beta-1} \\
&+ \beta(1-\beta) \left( \log \frac{r}{R} \right)^{\beta-2} 
\end{aligned} \right] cr^{-l-6} \\
& \qquad - \beta(n-2-2l) (l+2)(n-4-l) 
\left[ \begin{aligned}
& (l+4)(n-6-l)  \left( \log \frac{r}{R} \right)^{\beta-1}  \\
& +(1-\beta)(n-10-2l) \left( \log \frac{r}{R} \right)^{\beta-2}\\
&- (1-\beta)(2-\beta) \left( \log \frac{r}{R} \right)^{\beta-3} 
\end{aligned}\right] cr^{-l-6} \\
&\qquad +  \beta(1-\beta) \Big[ l(n-2-l)+(l+2)(n-4-l)\Big]\\
&\qquad \qquad \qquad \qquad  \times 
 \left[ \begin{aligned} & (l+4)(n-6-l)  \left( \log \frac{r}{R} \right)^{\beta-2} \\
& +(2-\beta)(n-10-2l) \left( \log \frac{r}{R} \right)^{\beta-3} \\
&- (2-\beta)(3-\beta) \left( \log \frac{r}{R} \right)^{\beta-4} 
\end{aligned}\right] cr^{-l-6} \\
&\qquad+\beta(1-\beta)(2-\beta) (n-2-2l)
\left[ \begin{aligned}
& (l+4)(n-6-l)  \left( \log \frac{r}{R} \right)^{\beta-3}  \\
& +(3-\beta)(n-10-2l) \left( \log \frac{r}{R} \right)^{\beta-4}\\
&- (3-\beta)(4-\beta) \left( \log \frac{r}{R} \right)^{\beta-5} 
\end{aligned}\right] cr^{-l-6} \\
& \qquad - \beta(1-\beta)(2-\beta)(3-\beta) 
\left[ \begin{aligned} 
& (l+4)(n-6-l)  \left( \log \frac{r}{R} \right)^{\beta-4} \\
& +(4-\beta)(n-10-2l) \left( \log \frac{r}{R} \right)^{\beta-5}\\
&- (4-\beta)(5-\beta) \left( \log \frac{r}{R} \right)^{\beta-6}
\end{aligned} \right] cr^{-l-6}.\\
&= A_0 r^{-m-6} - A_1 r^{-l-6} \log \frac{r}{R} + A_2 r^{-l-6} \\
&\quad+\left[ \begin{aligned} 
&B_0  \left( \log \frac{r}{R} \right)^{\beta} - B_1  \left( \log \frac{r}{R} \right)^{\beta-1}+ B_2  \left( \log \frac{r}{R} \right)^{\beta-2}   + B_3  \left( \log \frac{r}{R} \right)^{\beta-3}\\
&-B_4  \left( \log \frac{r}{R} \right)^{\beta-4} + B_5  \left( \log \frac{r}{R} \right)^{\beta-5}+ B_6  \left( \log \frac{r}{R} \right)^{\beta-6} 
\end{aligned} \right]cr^{-l-6},
\end{align*}
where $A_i$ with $i=0,1,2$ are given as follows
\[
\begin{split}
&A_0 = m(m+2)(m+4)(n-2-m)(n-4-m)(n-6-m),\\
&A_1 =B_0= l(l+2)(l+4)(n-2-l)(n-4-l)(n-6-l),\\
&A_2 =  (l+2)(n-4-l) \Big[l(n-2-l)(n-10-2l)+(l+4)(n-6-l)(n-2-2l)\Big],
\end{split}
\]
and $B_i$ with $i=1,\dots,6$ are given as follows
\[
\begin{split}
&B_1= \beta A_2,\\
&B_2= \beta(1-\beta) \left[
\begin{split} 
&l(n-2-l)(l+2)(n-4-l) \\
&- (l+2)(n-4-l)(n-2-2l)(n-10-2l) \\
&+(l+4)(n-6-l) \big[l(n-2-l)+(l+2)(n-4-l)\big]
\end{split}
\right],\\
&B_3= \beta(1-\beta)(2-\beta) 
 \left[
\begin{split} 
&(n-2-2l)(l+2)(n-4-l)  \\
&+(n-10-2l) \Big[ l(n-2-l)+(l+2)(n-4-l)\Big]\\
&+(n-2-2l)(l+4)(n-6-l) 
\end{split}
\right],\\
&B_4 = \beta(1-\beta)(2-\beta)(3-\beta)
\left[
\begin{split}
 &(l+2)(n-4-l)+l(n-2-l) \\
&-(n-2-2l)(n-10-2l)+ (l+4)(n-6-l)
\end{split}
 \right],\\
&B_5 =  \beta(1-\beta)(2-\beta)(3-\beta)(4-\beta)\left[-(n-2-2l)-(n-10-2l)\right],\\
&B_6 = \beta(1-\beta)(2-\beta)(3-\beta)(4-\beta)(5-\beta).
\end{split}
\]
Recall that $l=(n-6)/2$ and $\beta \in(0,1)$. Therefore, it is easy to see that
\[
A_2=B_1=B_3= B_5=0  
\] that 
\[ B_2 = \frac{1}{16}\beta(1-\beta) (3n^4-24n^3+72n^2-96n+304)\]
and that
\[
B_4 = \frac{3}{4} \beta(1-\beta)(2-\beta)(3-\beta) (3n^2-12n+44).
\]
Observe that 
\begin{align*}
\frac{B_4}{\frac{1}{2}B_2} &= \frac{24(2-\beta)(3-\beta)(3n^2-12n+44)}{3n^4-24n^3+72n^2-96n+304} \cdot \left( \log \frac{r}{R} \right)^{-2}\\
&\leqslant\frac{24(2-\beta)(3-\beta)(3n^2-12n+44)}{3n^4-24n^3+72n^2-96n+304} \cdot \left( \log \frac{\overline r_1}{R} \right)^{-2}\\
& = \frac{24(2-\beta)(3-\beta)(3n^2-12n+44)}{3n^4-24n^3+72n^2-96n+304} \cdot \left( \frac{c_1}{b} \right)^{-\frac{2}{1-\beta}}\\
&<1.
\end{align*}
Hence, $-\Delta \overline w_{\rm out}(r)$ is bounded from below by the function
\begin{equation}\label{BoundedBelow}
\begin{split}
&L^p r^{-m-6} - pL^{p-1}br^{-l-6} \left( \log \frac{r}{R} \right) \\
&+ pL^{p-1}  \left( \log \frac{r}{R} \right)^{\beta} cr^{-l-6} +\frac{1}{2} B_2 \left( \log \frac{r}{R} \right)^{\beta-2}  cr^{-l-6}.
\end{split}
\end{equation}
In order to apply Lemma \ref{LemmaBinomialIne}, we must have $0<\overline u_{\rm out} <Lr^{-m}$ for all $r>\overline r_1$. Indeed, set 
\[
f(r) : = L r^{l-m}- b \log \frac{r}{R}
\]
when $ r>\overline r_1$. A direct computation shows that
\[
f'(r) = L r^{l-m-1} -\frac{b}{r} > \frac{1}{r} \left( L \overline r_1^{l-m} -b \right) >0,
\] 
 which yields that 
\[
\overline u_{out}(r) > r^{l} f(r) > \overline r_1^l f(\overline r_1)> 0 \quad \text{for all} \quad r> \overline r_1.
\]
Furthermore, it can be easily seen that 
\[
b\log \frac{r}{R} > c \left( \log \frac{r}{R} \right)^{\beta}
\]
 for all $r >\overline r_1$. Hence, we claim that $0<\overline u_{out}(r) < Lr^{-m}$.  Thanks to \eqref{BinomialIne}, we get 
\begin{equation}\label{BoundedAbove}
\begin{split}
\overline u_{\rm out}^p &= (Lr^{-m})^p \Big[ 1- \left( \frac{b}{L}r^{m-l}\log \frac{r}{R} -\frac{c}{L} r^{m-l} \left( \log \frac{r}{R} \right)^{\beta} \right)\Big]^p\\
&\leqslant  (Lr^{-m})^p 
\left[ \begin{split}
& 1- \left( \frac{b}{L}r^{m-l}\log \frac{r}{R} -\frac{c}{L} r^{m-l} \left( \log \frac{r}{R} \right)^{\beta} \right) \\
&+ C_p \left( \frac{b}{L}r^{m-l}\log \frac{r}{R} -\frac{c}{L} r^{m-l} \left( \log \frac{r}{R} \right)^{\beta} \right)^2 
 \end{split} \right]\\
&\leqslant L^p r^{-m-6} - pL^{p-1}br^{-l-6} \log \frac{r}{R}+ pL^{p-1}cr^{-l-6} \left( \log \frac{r}{R} \right)^{\beta} \\
& \qquad\qquad+C_p L^{p-2}b^2 r^{m-2l-6}  \left( \log \frac{r}{R} \right)^2
\end{split}
\end{equation} 
for $r >\overline r_1$. Due to \eqref{BoundedBelow} and \eqref{BoundedAbove} , we obtain
\begin{align*}
-\Delta \overline w_{\rm out} - \overline u_{\rm out}^p &\geqslant \frac{\beta(1-\beta) (3n^4-24n^3+72n^2-96n+304)}{16} \cdot cr^{-l-6} \left( \log \frac{r}{R} \right)^{\beta-2}\\
&\qquad \qquad - C_p L^{p-2}b^2r^{m-2l-6}\left( \log \frac{r}{R} \right)^{2},
\end{align*}
for $r\geqslant \overline r_1$. We can see that the right-hand side of the previous inequality has the same sign with that of 
\[
g(r) - \frac{16C_p L^{p-2}b^2}{c\beta(1-\beta)(3n^4-24n^3+72n^2-96n+304)},
\]
where
\[
g(r):= r^{l-m} \left( \log \frac{r}{R} \right)^{\beta-4}.
\]
Computing directly, for all $r\geqslant \overline r_1$, we know that
\begin{align*}
g'(r) &= r^{l-m-1} \left( \log \frac{r}{R} \right)^{\beta-5} \Big[ (l-m)\log \frac{r}{R} +\beta -4\Big]\\
& > r^{l-m-1} \left( \log \frac{r}{R} \right)^{\beta-5} \Big[  (l-m)\log \frac{\overline r_1}{R} +\beta -4\Big]\\
&= r^{l-m-1} \left( \log \frac{ r}{R} \right)^{\frac{1}{1-\beta}} \Big[  (l-m)\left( \frac{c}{b}\right)^{\frac{1}{1-\beta}} +\beta -4\Big]\\
& >0.
\end{align*}
Hence, 
\[
g(r) > g(\overline r_1) > C_p L^{p-2}b^2r^{m-2l-6}\left( \log \frac{r}{R} \right)^{2}
\]
due to the choice of $c$. This is equivalent to saying that
\[
-\Delta \overline w_{\rm out} \geqslant \overline u_{\rm out}^p
\]
for all $r\geqslant \overline r_1$. This establishes \eqref{OuterSuperSolutionCase2} as claimed.
\end{proof}
The following result is  an analouge of Lemma \ref{LemmaBorderline}.
\begin{lemma}\label{BorderlineEps}
Let $\beta \in (0,1)$ and  $\overline r_1$ be given in Lemma \ref{LemOuterSupSol}. Then, there exists $c_2>0$ such that  given any $c>c_2$, we can select positive real numbers $\overline r_2, \overline r_3 $ and $\epsilon_0$ such that for all $\epsilon \in (0,\epsilon_0)$, there exists  numbers $\overline r_1(\epsilon), \overline r_2(\epsilon), \overline r_3(\epsilon)$ satisfying
\begin{equation}
\overline r_1(\epsilon) := \sup\{r>0 : \overline u_{\rm in,\epsilon} < \overline u_{\rm out} \text{ in } (0,r)\} \in (0,+\infty],
\end{equation}
\begin{equation}
\overline r_2(\epsilon) := \sup\{r>0 : \overline v_{\rm in,\epsilon} < \overline v_{\rm out} \text{ in } (0,r)\} \in (0,+\infty],
\end{equation}
and
\begin{equation}
\overline r_3(\epsilon) := \sup\{r>0 : \overline w_{\rm in,\epsilon} < \overline w_{\rm out} \text{ in } (0,r)\} \in (0,+\infty],
\end{equation}
as well as
\begin{equation}\label{OrderRover}
\overline r_1 < \overline r_1(\epsilon) < \overline r_2< \overline r_2(\epsilon)<\overline r_3 <\overline r_3(\epsilon) < \overline r_3+1.
\end{equation}
Here, the functions $\overline u_{\rm in,\epsilon}, \overline v_{\rm in,\epsilon}, \overline w_{\rm in,\epsilon}, \overline u_{\rm out}, \overline v_{\rm out}, \overline w_{\rm out}$ are defined through \eqref{InnerSuperSolution} and \eqref{OuterSuperSolutionCase2}.
\end{lemma}
\begin{proof}
Let $c_2$ be chosen such that
\[
\left\{
\begin{split}
&l(n-2-l)b-\beta(n-2-2l)b^{\frac{2-\beta}{1-\beta}} c_2^{-\frac{1}{1-\beta}} >0,\\
&l(n-2-l) \left(\frac{c_2}{b}\right)^{\frac{2}{1-\beta}}
 - (2-\beta)(3-\beta) >0,
\end{split}
\right.
\]
and that 
\[
\begin{split}
&l(l+2)(n-2-l)(n-4-l)b - \beta  \Big[l(n-2-l)+(l+2)(n-4-l)\Big]  b^{\frac{3-\beta}{1-\beta}} c_2^{-\frac{2}{1-\beta}}\\
&\qquad \qquad-  \beta(2-\beta) (n-2-2l)  b^{\frac{4-\beta}{1-\beta}} c_2^{-\frac{3}{1-\beta}} >0.
\end{split}
\]
Let us denote the following gap functions,
\[
\left\{
\begin{split}
\varphi_{1,\epsilon}(r) &:= \overline u_{\rm in,\epsilon}(r) -\overline u_{\rm out}(r),\\
\varphi_{2,\epsilon}(r) &:= \overline v_{\rm in,\epsilon}(r) -\overline v_{\rm out}(r), \\
\varphi_{3,\epsilon}(r) &:= \overline w_{\rm in,\epsilon}(r) -\overline w_{\rm out}(r). 
\end{split}
\right.
\]
Then as $\epsilon \searrow 0$, we have in $C_{loc}^1((R,\infty))$ that
\begin{equation}\label{ConvPhi1}
\varphi_{1,\epsilon}(r) \nearrow \varphi_1(r) := b r^{-l} \log \frac{r}{R} - cr^{-l}
\left( \log \frac{r}{R}\right)^{\beta}, 
\end{equation}
that
\begin{equation}\label{ConvPhi2}
\begin{split}
\varphi_{2,\epsilon}(r) \nearrow \varphi_2(r) &:= \Big[ l(n-2-l) \log \frac{r}{R} - (n-2-2l)\Big] br^{-l-2} \\
&- \left[ \begin{split}
&   l(n-2-l) \left( \log \frac{r}{R} \right)^{\beta}\\
& -\beta(n-2-2l)  \left( \log \frac{r}{R} \right)^{\beta-1} \\
& +\beta(1-\beta)  \left( \log \frac{r}{R} \right)^{\beta-2} 
\end{split} \right]
cr^{-l-2},
\end{split}
\end{equation}
and that
\[ \varphi_{3,\epsilon}(r) \nearrow \varphi_3(r) \]
where
\begin{equation}\label{ConvPhi3}
\begin{split}
 \varphi_3(r) &:= (l+2)(n-4-l) \Big[  l(n-2-l) \log \frac{r}{R}  -(n-2-2l) \Big] br^{-l-4} \\
&- \left[
\begin{split}
 & l (l+2) (n-2-l) (n-4-l) \left( \log \frac{r}{R} \right)^{\beta} \\
&- \beta(n-2-2l) (l+2)(n-4-l) \left( \log \frac{r}{R} \right)^{\beta-1}\\
&+\beta(1-\beta)   [l(n-2-l)+(l+2)(n-4-l)] \left( \log \frac{r}{R} \right)^{\beta-2}\\
& + \beta(1-\beta)(2-\beta) (n-2-2l) \left( \log \frac{r}{R} \right)^{\beta-3}\\
& - \beta(1-\beta)(2-\beta)(3-\beta) \left( \log \frac{r}{R} \right)^{\beta-4} 
\end{split}
\right]cr^{-l-4}.
\end{split}
\end{equation}
Similarly to \eqref{Psi1}, \eqref{Psi2} and \eqref{Psi3}, we have to specify the sign of the limiting functions $\varphi_i$ with $i=1,2,3$. First, it follows from the choice of $\overline r_1$ stated in Lemma \ref{LemOuterSupSol} that
\[ \varphi_1
\begin{cases}
<0 \quad&\text{for all } r<\overline r_1,\\
=0 \quad&\text{for } r=\overline r_1,\\
>0 \quad&\text{for all } r>\overline r_1.
\end{cases}
\]
Now a direct calculation shows that
\[
\begin{split}
\overline r_1^{l+2} \varphi_2(\overline r_1) 
&= \Big[ l(n-2-l) \left( \frac{c}{b}\right)^{\frac{1}{1-\beta}} - (n-2-2l) \Big] b \\
& - \Big[ l(n-2-l) \left( \frac{c}{b}\right)^{\frac{\beta}{1-\beta}} -\beta(n-2-2l)  \left( \frac{c}{b} \right)^{-1} +\beta(1-\beta)  \left( \frac{c}{b} \right)^{\frac{\beta-2}{1-\beta}} \Big]c \\
&=-(1-\beta)(n-2-2l)b -\beta(1-\beta)^{\frac{2-\beta}{1-\beta}} c^{-\frac{1}{1-\beta}}\\
&<0,
\end{split}
\]
and that
\[ \begin{split}
&\frac{d}{dr} \left( r^{l+2} \varphi_2(r) \right)
 = l(n-2-l) \frac{b}{r} -
\left[ \begin{split}
&  \beta l(n-2-l) \left( \log \frac{r}{R} \right)^{\beta-1}\\
& + \beta(1-\beta) (n-2-2l)  \left( \log \frac{r}{R} \right)^{\beta-2}\\
 &- \beta(1-\beta)(2-\beta)  \left( \log \frac{r}{R} \right)^{\beta-3} 
\end{split} \right]
 \frac{c}{r}\\
&\geqslant \frac{1}{r} \left[ \begin{split}
&\beta l(n-2-l) b - (1-\beta) (n-2-2l)   \left( \frac{c}{b} \right)^{-1} c  \\
& - \beta(1-\beta)  (n-2-2l)  \left( \frac{c}{b} \right)^{\frac{\beta-2}{1-\beta}} c
\end{split} \right]  \\
& =\frac{1}{r} \Big[ (1-\beta) l(n-2-l) b - \beta(1-\beta) (n-2-2l) b^{\frac{2-\beta}{1-\beta}}c^{-\frac{1}{1-\beta}} \Big]\\
&>0.
\end{split}\]
Here we have used the fact that $\beta \in(0,1)$ and that $2l=n-6 < n-2$. These estimates implies that there exists some $\overline r_2 > \overline r_1$ such that
\[ \varphi_2
\begin{cases}
<0 \quad&\text{for all } r<\overline r_2,\\
=0 \quad&\text{for } r=\overline r_2,\\
>0 \quad&\text{for all } r>\overline r_2.
\end{cases}
\]
%
Now we consider the limiting function $\varphi_3$ The fact that $\varphi_2$ has only one positive root $\overline r_2$ leads us to 
\begin{equation*}
\begin{split}
&\Big[ l(n-2-l) \log \frac{\overline r_2}{R} - (n-2-2l)\Big] b = \left[ \begin{split}
& l(n-2-l) \left( \log \frac{r}{R} \right)^{\beta} \\
& -\beta(n-2-2l)  \left( \log \frac{r}{R} \right)^{\beta-1}\\
& +\beta(1-\beta)  \left( \log \frac{r}{R} \right)^{\beta-2} 
\end{split} \right]c
\end{split}
\end{equation*}
We subsitute into $\varphi_3$ to get  
\[\begin{split}
\varphi_3(\overline r_2) = - c \beta(1-\beta) \left( \log \frac{\overline r_2}{R} \right)^{\beta -4} 
\left[ \begin{split} 
& l(n-2-l) \left( \log \frac{\overline r_2}{R} \right)^2\\
& + (2-\beta)(n-2-2l)  \log \frac{\overline r_2}{R}\\
& - (2-\beta)(3-\beta) 
\end{split} \right].
\end{split}\]
Thanks to $\overline r_2 > \overline r_1$, we have $\varphi_3(\overline r_2) <0$. Furthermore, a direct computation shows that 
\[\begin{split}
r \cdot \frac{d}{dr} \left( r^{l+4} \varphi_3(r) \right) &= l(l+2)(n-2-l)(n-4-l)b \\
&- \left[ 
\begin{split}
& l (l+2) (n-2-l) (n-4-l) \left( \log \frac{r}{R} \right)^{\beta-1} \\
& - \beta(n-2-2l) (l+2)(n-4-l) \left( \log \frac{r}{R} \right)^{\beta-2}   \\
&  +\beta(1-\beta)   \Big[l(n-2-l)+(l+2)(n-4-l)\Big] \left( \log \frac{r}{R} \right)^{\beta-3}  \\
&  + \beta(1-\beta)(2-\beta) (n-2-2l) \left( \log \frac{r}{R} \right)^{\beta-4} \\
&  - \beta(1-\beta)(2-\beta)(3-\beta) \left( \log \frac{r}{R} \right)^{\beta-5} 
\end{split} \right]c .
\end{split} \]
Since $\beta \in(0,1)$, for $r > \overline r_2 ( > \overline r_1)$, the  right-hand side  is bounded from below by 
\[ \begin{split}
 &l(l+2)(n-2-l)(n-4-l)b \\
&- \left[
\begin{split} &l (l+2) (n-2-l) (n-4-l) \left(\frac{c}{b} \right)^{-1} \\
&  - \beta(n-2-2l) (l+2)(n-4-l) \left(  \frac{c}{b} \right)^{\frac{\beta-2}{1-\beta}}  \\
&  +\beta(1-\beta)   [l(n-2-l)+(l+2)(n-4-l)]  \left(  \frac{c}{b} \right)^{\frac{\beta-3}{1-\beta}} \\
&  + \beta(1-\beta)(2-\beta) (n-2-2l) \left(  \frac{c}{b} \right)^{\frac{\beta-4}{1-\beta}}  \\
&   - \beta(1-\beta)(2-\beta)(3-\beta) \left(  \frac{c}{b} \right)^{\frac{\beta-5}{1-\beta}} 
\end{split} \right]c,
\end{split} \]
which is larger than
\begin{align*}
& (1-\beta)l(l+2)(n-2-l)(n-4-l)b \\
&- \beta(1-\beta)   [l(n-2-l)+(l+2)(n-4-l)]  b^{\frac{3-\beta}{1-\beta}} c^{-\frac{2}{1-\beta}}\\
&-  \beta(1-\beta)(2-\beta) (n-2-2l)  b^{\frac{4-\beta}{1-\beta}} c^{-\frac{3}{1-\beta}} >0.
\end{align*}
Hence, the function $\varphi_3$ admits a unique zero point $\overline r_3$, which immediately yields
\[\varphi_3
\begin{cases}
<0 \quad&\text{for all } r<\overline r_3,\\
=0 \quad&\text{for } r=\overline r_3,\\
>0 \quad&\text{for all } r>\overline r_3.
\end{cases}
\]
Having the sign of all limiting function $\varphi_i$ with $i=1,2,3$ we can repeat the argument used in the proof of Lemma \ref{LemmaBorderline} to conclude the lemma.
\end{proof}

We are now in position to construct a super-solution to \eqref{LaneEmdenSystem}.
\begin{lemma}\label{LemmaSuperSolutionCase2}
Let $p=\pc$.
Then there exists $k>l,c>0$ and $\epsilon_0 >0$ such that whenever $\epsilon \in (0,\epsilon_0)$, the functions
\begin{equation} 
\overline u_{\epsilon}(x) := 
\begin{cases}
\overline u_{\rm in,\epsilon}(|x|), \quad & 0\leqslant |x|<\overline r_1(\epsilon),\\
\overline u_{\rm out}(|x|), \quad & |x|\geqslant \overline r_1(\epsilon),
\end{cases}
\end{equation}
\begin{equation} 
\overline v_{\epsilon}(x) := 
\begin{cases}
\overline v_{\rm in,\epsilon}(|x|), \quad & 0\leqslant |x|<\overline r_2(\epsilon),\\
\overline v_{\rm out}(|x|), \quad & |x|\geqslant \overline r_2(\epsilon). 
\end{cases}
\end{equation}
\begin{equation} 
\overline w_{\epsilon}(x) := 
\begin{cases}
\overline w_{\rm in,\epsilon}(|x|), \quad & 0\leqslant |x|<\overline r_3(\epsilon),\\
\overline w_{\rm out}(|x|), \quad & |x|\geqslant \overline r_3(\epsilon), 
\end{cases}
\end{equation}
are nonnegative, Lipschitz continuous on $\R^n$ and satisfy 
\begin{equation}\label{SuperSolutionCase2}
\begin{cases}
-\Delta \overline u_{\epsilon} \geqslant  \overline v_{\epsilon},  \\
-\Delta \overline v_{\epsilon} \geqslant \overline w_{\epsilon}, \\
-\Delta \overline w_{\epsilon} \geqslant \overline u_{\epsilon}^p
\end{cases}
\end{equation}
in the distributional sense on $\R^n$.
\end{lemma}

The proof of Lemma \ref{LemmaSuperSolutionCase2} is similar to that Lemma \ref{LemmaSuperSolution}; hence we omit its details.

\section{Proof of Theorem  \ref{MainThmSupercriticalCase} and \ref{MainThmCriticalCase}}

\subsection{Comparision principle and monotonicity property for parabolic systems of cooperative type.} One of key ingredients in proving Theorem \ref{MainThmSupercriticalCase} and \ref{MainThmCriticalCase} is a comparision principle and monotonicity property for parabolic systems of cooperative type. 
\begin{lemma}\label{ComparisionPrinciple}
Let $0<T<\infty$ and $\Omega$ be an arbitrary domain (bounded or unbounded) in $\R^n$,  we denote that $\Omega_T :=\Omega \times (0,T)$ and $\P_T$ is parabolic boundary of  $\Omega_T$.  For $i=1,2,3$, let $f_i =f_i(u_1,u_2,u_3) : \R^3 \to \R$ be $C^1$--functions such that $\partial_{u_i}f_j \geqslant 0$ for $i\neq j$ and  $u_i,v_i$ be functions such that 
\[
u_i, v_i  \in  C(\overline \Omega \times (0,T))\cap C( [0,T), L_{loc}^2(\overline \Omega)) \cap L^{\infty}(\Omega_T),
\] 
as well as
\[
\partial_t u_i , \nabla u_i  , \nabla^2 u_i  \in L_{loc}^2(\Omega_T), \quad \text{and} \quad \partial_t v_i , \nabla v_i  , \nabla^2 v_i  \in L_{loc}^2(\Omega_T).
\]
 We suppose that $u_i \leqslant v_i$ on $\P_T$ and 
\[
\partial_t u_i -d_i \Delta u_i -f_i(u_1,u_2,u_3) \leq \partial_t v_i -d_i \Delta v_i -f_i(v_1,v_2,v_3) 
\]
a.e. in $\Omega_T$, then 
\[
u_i \leqslant v_i \quad \text{in } \Omega_T,
\]
for all  $ i=1,2,3$.
\end{lemma}

Lemma \ref{ComparisionPrinciple} above give us the comparision principle for parabolic systems of cooperative type, whose proof is standard and followed by \cite[Prop 52.21, 52.22]{QS07}.
\begin{lemma}\label{MonotoneSolution}
We consider the following system
\begin{equation}\label{CoParabolicSystem}
\left\{
\begin{split}
\partial_t u_i - d_i \Delta u_i &= f_i(u_1,u_2,u_3), \quad x\in \R^n,t>0, i=1,2,3,\\
u_i(x,0)&=v_{i}(x), \quad \text{for } i=1,2,3,
\end{split}
\right.
\end{equation}
where $f_1,f_2,f_3$ are $C^1$-functions such that 
\begin{equation}\label{AssumeFi}
\partial_{u_j}f_i\geqslant  0\quad \text{for } i\neq j.
\end{equation}
For $i=1,2,3$ the intial data $v_{i} \in L^{\infty}(\R^n)$ satisfies 
\begin{equation}\label{WeakData}
\Delta v_{i} + f_i(v_{1},v_{2},v_{3}) \geqslant 0
\end{equation}
 in the very weak sense if, for all nonnegative function $\psi \in C^2(\R^n)$, there holds 
\begin{equation}\label{WeakSense}
\int_{\R^n} \Big[ v_{i} \Delta \psi +f_i(v_{1},v_{2},v_{3}) \psi \Big] dx \geqslant 0.
\end{equation}
Therefore, we claim that for system \eqref{CoParabolicSystem} under assumption \eqref{AssumeFi}, and $v_1, v_2$ and $v_3$ satisfy \eqref{WeakData} in the very weak sense,  then 
\[ \partial_t u_i \geqslant 0\]
in  $\R^n \times (0,\infty)$ for all  $i=1,2,3$.
\end{lemma}

 Lemma \ref{MonotoneSolution} above can be seen as an extension of \cite[Prop 52.20]{QS07}, which was stated in bounded domain. Similar arguments  still guarantee the monotonicity of solutions $u_1, u_2$ and $u_3$ along time in the whole space.

\subsection{Proof of Theorem \ref{MainThmSupercriticalCase}}
Before proving Theorem \ref{MainThmSupercriticalCase}, we need to establish the following result.
\begin{lemma}
Let $p>\pc$ and  $\epsilon_0, \overline r_i, \overline r_i(\epsilon)  (i=1,2,3)$ be stated in Lemma \ref{LemmaBorderline}. Two triplets ($\overline u_{\rm in,\epsilon},\overline v_{\rm in,\epsilon}, \overline w_{\rm in,\epsilon} $) and ($\underline u_{\rm out}, \underline v_{\rm out}, \underline w_{\rm out}$) are defined through \eqref{InnerSuperSolution} and \eqref{OuterSubSolution}. Then, the functions $\underline u, \underline v, \underline w$ and $\overline u_{\epsilon}, \overline v_{\epsilon}, \overline w_{\epsilon}$ appeared in Lemma \ref{LemmaSubSolution} and Lemma \ref{LemmaSuperSolution} satisfy
\begin{equation}\label{OrderSolution}
\underline u \leqslant \overline u_{\epsilon} ,\quad 
\underline v \leqslant \overline v_{\epsilon} \quad \text{and} \quad
\underline w \leqslant \overline w_{\epsilon} 
\end{equation}
for small enough $\epsilon>0$  in $\R^n$. 
\end{lemma}
\begin{proof}
Since $\overline u_{\epsilon}$ is always nonnegative, there is a need to prove the first inequality of \eqref{OrderSolution} as $|x| \geqslant \underline r_1$. In the case $|x| \geqslant \overline r_1(\epsilon)$, we have $\underline u(x) = \underline u_{\rm out}(|x|)$. Hence,
\[
\overline u_{\epsilon}(x) =  \overline u_{\rm out}(|x|)= L|x|^{-m}-b|x|^{-l}+c|x|^{k} \geqslant  L|x|^{-m}-b|x|^{-l} = \underline u(x).
\]
When $\underline r_1 \leqslant |x| <\overline r_1(\epsilon)$, due to $m<l$ we have
\begin{equation}\label{IneOrder}
\overline u_{\epsilon}(x) = \overline u_{\rm in,\epsilon}(|x|) = L(|x|^2+\epsilon)^{-\frac{m}{2}} \geqslant L|x|^{-m} -b|x|^{-l} = \underline u(x).
\end{equation}
These establish $\underline u \leqslant \overline u_{\epsilon}$. The others of \eqref{OrderSolution} can be proved similarly.
\end{proof}

We are now in position to prove Theorem \ref{MainThmSupercriticalCase}.
Let the triplets $(\underline u, \underline v, \underline w)$ and ($\overline u_{\epsilon}, \overline v_{\epsilon}, \overline w_{\epsilon})$ are defined as in Lemma \ref{LemmaSubSolution} and \ref{LemmaSuperSolution}. 
We consider the parabolic system 
\begin{equation}\label{ParabolicSystem}
\left\{
\begin{split}
&U_t = \Delta U+V, \quad x\in \R^n, t>0,\\
&V_t=\Delta V+W, \quad x\in \R^n, t>0,\\
&W_t =\Delta W +|U|^{p-1}U, \quad x\in \R^n, t>0,\\
&U(x,0)=\underline u(x), V(x,0) = \underline v(x), W(x,0) = \underline w(x) \quad x\in \R^n. 
\end{split}
\right.
\end{equation}

It could be realized that the PDE system \eqref{ParabolicSystem} is of cooperative type and the initial datas are sub-solutions of \eqref{LaneEmdenSystem}. Hence, we apply  Lemma \ref{ComparisionPrinciple} for 
\[
f_1(x_1,x_2,x_3) = x_2, \quad f_2(x_1,x_2,x_3) = x_3 \quad \text{and } f_3(x_1,x_2,x_3) = |x_1|^{p-1}x_1
\]
to obtain that \eqref{ParabolicSystem} has a globally radially symmetric classical solution along with the following comparision principle,
\begin{equation*}
\underline u(x) \leqslant U(x,t) \leqslant \overline u_{\epsilon}(x),\quad \underline v(x) \leqslant V(x,t) \leqslant \overline v_{\epsilon}(x) \quad \text{ and } \quad \underline w(x) \leqslant W(x,t) \leqslant \overline w_{\epsilon}(x)
\end{equation*}
for all $x\in \R^n, t>0$. Moreover, 
from Lemma \ref{MonotoneSolution}, the initial data also implies the monotonicity in time of solution, i.e.
\[ 
U_t \geqslant 0, \quad V_t \geqslant 0, \quad \text{and} \quad W_t \geqslant 0 \]
in $\R^n \times (0,+\infty)$.
As a consequence of the  standard regularity theory, we  get
\[
U(\cdot,t) \to u, \quad V(\cdot, t) \to v, \quad \text{and} \quad W(\cdot,t) \to w 
\]
in $C_{loc}^2(\R^n)$  sense as $ t\to +\infty$, where $u,v$ and $w$ are smooth limit functions. Thanks to  the  monotocity of $U,V,W$ along time  again and parabolic regularity theory, there exists a sequence of times $t_j \to \infty$ such that $U_t(\cdot,t_j), V_t(\cdot,t_j), W_t(\cdot,t_j)$ converge to zero in $C_{loc}^0(\R^n)$ as $j \to +\infty$. At time $t=t_j$, we have
\begin{equation*}
\left\{
\begin{split}
U_t (x,t_j)&= \Delta U (x, t_j)+V(x, t_j), &\quad x\in \R^n,\\
V_t(x, t_j)&=\Delta V(x, t_j)+W(x, t_j), &\quad x\in \R^n,\\
W_t(x, t_j)&=\Delta W(x, t_j) +|U(x, t_j)|^{p-1}U(x, t_j), &\quad x\in \R^n.
\end{split}
\right.
\end{equation*}
Letting $j \to \infty$ to show that the limit functions $u,v,w$  is a globally radial solution of the following stationary problem
\begin{equation*}
\left\{
\begin{split}
0&=\Delta u + v, \\
0&=\Delta v +w,\\
0&=\Delta w + |u|^{p-1}u
\end{split}
\right.
\end{equation*}
in $\R^n$. We recall the comparision principle  to get that $u$ is strictly positive in $\R^n$. Hence, $u$ is a smooth positive solution of 
\[
(-\Delta)^3 u = u^p \quad \text{in } \R^n,
\]
which is bounded from below as 
\[
L|x|^{-m}-b|x|^{-l} \leqslant u(x) \leqslant L|x|^{-m}-b|x|^{-l}+c|x|^{-k},
\] 
for large $|x|$. Thus, Theorem  \ref{MainThmSupercriticalCase} is complete. 

\subsection{Proof of Theorem \ref{MainThmCriticalCase}.}
As in the supercritical case, we need to establish the following result for the critical case. 
\begin{lemma}\label{LemCompaSubSuper}
Let $p=\pc$ and  $\epsilon_0, \overline r_i, \overline r_i(\epsilon) (i=1,2,3)$ be found  in Lemma \ref{BorderlineEps}. The functions $\overline u_{\rm in,\epsilon},\overline v_{\rm in,\epsilon}, \overline w_{\rm in,\epsilon}$ and $\underline u_{\rm out}, \underline v_{\rm out}, \underline w_{\rm out}$ are defined through \eqref{InnerSuperSolution} and \eqref{OuterSubSolCase2}. Then, the triplets $(\underline u, \underline v, \underline w)$ and ($\overline u_{\epsilon}, \overline v_{\epsilon}, \overline w_{\epsilon})$ appeared in Lemma \ref{LemmaSubSolCase2} and Lemma \ref{LemmaSuperSolutionCase2} satisfy
\begin{equation*}
\underline u \leqslant \overline u_{\epsilon} ,\quad 
\underline v \leqslant \overline v_{\epsilon} \quad \text{and} \quad
\underline w \leqslant \overline w_{\epsilon} 
\end{equation*}
for small enough $\epsilon>0$ in $\R^n$. 
\end{lemma}
\begin{proof}
Let $c_1$ and $c_2$ be constants found in Lemma \ref{LemOuterSupSol} and \ref{BorderlineEps}. Besides, $R,\underline r_1,\underline r_2$ and $\underline r_3$ are stated in Lemma \ref{LemmaComparisionUout}. Choosing large enough constant such that  $c > \max \{c_1,c_2\}$ and that 
\[
\left(\frac{c}{b} \right)^{\frac{1}{1-\beta}} > \log \frac{ \underline r_3}{R}, \qquad l(n-2-l)  \left( \frac{c}{b} \right)^{\frac{1}{1-\beta}} > \beta(n-2-2l),
\]
and that 
\begin{align*}
 l (l+2) (n-2-l) (n-4-l) \left( \frac{c}{b} \right)^{\frac{4}{1-\beta}} 
&-  \beta(n-2-2l) (l+2)(n-4-l) \left( \frac{c}{b} \right)^{\frac{3}{1-\beta}} \\
& > \beta(1-\beta)(2-\beta)(3-\beta).
\end{align*}
Due to the choice $c$, it is not hard to see that  $\underline r_3 < \overline r_1$.
Hence, \eqref{OrderRunder} and  \eqref{OrderRover} can be joined into 
\[
\underline r_1 < \underline r_2 < \underline r_3 < \overline r_1 < \overline r_1(\epsilon) < \overline r_2 < \overline r_2(\epsilon)< \overline r_3 < \overline r_3(\epsilon) < \overline r_3 +1. 
\]
That implies, 
\begin{equation} \overline u_{\epsilon}(x) - \underline u(x) = 
\begin{cases}
\overline u_{\rm in,\epsilon}(|x|),  \quad&  0\leqslant |x| < \underline r_1,\\ 
\overline u_{\rm in,\epsilon}(|x|) - \underline u_{\rm out}(|x|), \quad&  \underline r_1\leqslant |x| < \overline r_1(\epsilon),\\ 
\overline u_{\rm out}(|x|) - \underline u_{\rm out}(|x|),  \quad&  \overline r_1(\epsilon) \leqslant |x|, 
\end{cases}
\end{equation}
\begin{equation} \overline v_{\epsilon}(x) - \underline v(x) = 
\begin{cases}
\overline v_{\rm in,\epsilon}(|x|),  \quad&  0\leqslant |x| < \underline r_2,\\ 
\overline v_{\rm in,\epsilon}(|x|) - \underline v_{\rm out}(|x|), \quad&  \underline r_2\leqslant |x| < \overline r_2(\epsilon),\\ 
\overline v_{\rm out}(|x|) - \underline v_{\rm out}(|x|),  \quad&  \overline r_2(\epsilon) \leqslant |x|, 
\end{cases}
\end{equation}
and 
\begin{equation} \overline w_{\epsilon}(x) - \underline w(x) = 
\begin{cases}
\overline w_{\rm in,\epsilon}(|x|),  \quad&  0\leqslant |x| < \underline r_3,\\ 
\overline w_{\rm in,\epsilon}(|x|) - \underline w_{\rm out}(|x|), \quad&  \underline r_3\leqslant |x| < \overline r_3(\epsilon),\\ 
\overline w_{\rm out}(|x|) - \underline w_{\rm out}(|x|),  \quad&  \overline r_3(\epsilon) \leqslant |x|. 
\end{cases}
\end{equation}
A repetition to \eqref{IneOrder} deduce that 
\[
\underline u_{\rm out}(r) < L (r^2+\epsilon)^{-\frac{m}{2}} 
\]
for all  $r \in [\underline r_1, \overline r_3+1]$,  that
\[
\underline v_{\rm out}(r) < m(n-2-m)L (r^2+\epsilon)^{-\frac{m+2}{2}}
\] 
for all $ r \in [\underline r_2, \overline r_3+1]$ and that  
\[
\underline w_{\rm out}(r) < m(m+2)(n-2-m)(n-4-m)L (r^2+\epsilon)^{-\frac{m+4}{2}}
\] for all $r \in [\underline r_3, \overline r_3+1]$. That means $\overline u_{\epsilon} - \underline u_{\rm out}, \overline v_{\epsilon} - \underline v_{\rm out} $ and $\overline w_{\epsilon} - \underline w_{\rm out}$ have a positive sign in the interval $[\overline r_1, \overline r_3+1], [\overline r_2, \overline r_3+1]$ and $[\overline r_3, \overline r_3+1]$, respectively. 

In order to complete Lemma \ref{LemCompaSubSuper}, we have to show that $\overline u_{\rm out}- \underline u_{\rm out}, \overline v_{\rm out}- \underline v_{\rm out}$ and $\overline w_{\rm out}- \underline w_{\rm out}$ are nonnegative in the correspondind outer domain. 
First, we obviously see that
\[
\overline u_{\rm out}(r) - \underline u_{\rm out}(r) = c r^{-l} \left(\log \frac{r}{R} \right)^{\beta} >0, 
\]
for $r > R$.  Besides, we also have that
\[\begin{split}
&\overline v_{\rm out}(r) - \underline v_{\rm out}(r) \\
&= c r^{-l-2} \left[ \begin{split}
&l(n-2-l)  \left(\log \frac{r}{R} \right)^{\beta}\\
& -\beta(n-2-2l)  \left(\log \frac{r}{R} \right)^{\beta-1} \\
& +\beta(1-\beta)  \left(\log \frac{r}{R} \right)^{\beta-2} \end{split} \right]\\
&> cr^{-l-2}  \left(\log \frac{r}{R} \right)^{\beta-1}  \Big[ l(n-2-l)  \log\frac{\overline r_1}{R} -\beta(n-2-2l) \Big] \\
&>0
\end{split}\]
for all $r>\overline r_1$. Next, we need to verify that $\overline w_{\rm out} - \underline w_{\rm out}$ does not change sign in $(\overline r_3(\epsilon), \infty)$. Indeed,
\begin{align*}
&\overline w_{\rm out}(r) -\underline w_{\rm out}(r) \\
&= cr^{-l-4}
\left[ \begin{aligned}
&  l (l+2) (n-2-l) (n-4-l) \left( \log \frac{r}{R} \right)^{\beta} \\
&  - \beta(n-2-2l) (l+2)(n-4-l) \left( \log \frac{r}{R} \right)^{\beta-1} \\
& +\beta(1-\beta)   (l(n-2-l)+(l+2)(n-4-l)) \left( \log \frac{r}{R} \right)^{\beta-2}  \\
&  + \beta(1-\beta)(2-\beta) (n-2-2l) \left( \log \frac{r}{R} \right)^{\beta-3} \\
&   - \beta(1-\beta)(2-\beta)(3-\beta) \left( \log \frac{r}{R} \right)^{\beta-4} 
\end{aligned} \right] \\
&> cr^{-l-4} \left(\log \frac{r}{R} \right)^{\beta-4} 
\left[ \begin{aligned}
& l (l+2) (n-2-l) (n-4-l) \left(\log \frac{\overline r_1}{R} \right)^4 \\
 &- \beta(n-2-2l) (l+2)(n-4-l) \left( \log \frac{\overline r_1}{R} \right)^3\\
& - \beta(1-\beta)(2-\beta)(3-\beta) 
\end{aligned} \right]\\
&>0.
\end{align*}
Hence, we complete proof of Lemma \ref{LemCompaSubSuper}.
\end{proof}

Theorem \ref{MainThmCriticalCase} can be deduced by copying proof of Theorem \ref{MainThmSupercriticalCase} with a simple modification. The super-solutions $\overline u_{\epsilon},\overline v_{\epsilon}, \overline w_{\epsilon}$ are taken from Lemma \ref{LemmaSuperSolutionCase2}, not from Lemma \ref{LemmaSuperSolution}.

\section{Acknowledgments}
The author  would like to  thank Qu\^oc Anh Ng\^o for suggesting this problem. Thanks also go to Ninh Van Thu for his encouragement.

\end{document}